\documentclass{amsart}
\usepackage[utf8]{inputenc}
\usepackage{array}
\usepackage{ifthen}
\usepackage{url}
\usepackage{hyperref}
\usepackage{amssymb,amsmath, amsfonts}
\newif\iffull
\usepackage{xcolor}
\usepackage[top=2cm, bottom=2cm, left=01cm, right=1cm]{geometry}

\providecommand{\keywords}[1]{\textbf{\textit{Mots-cl\'es :}} #1}

\newtheorem{theorem}{Theorem}[section]
\newtheorem{lemma}[theorem]{Lemma}
\newtheorem{proposition}[theorem]{Proposition}
\newtheorem{cor}[theorem]{Corollary}

\theoremstyle{definition}

\allowdisplaybreaks
\theoremstyle{remark}

\numberwithin{equation}{section}

\newcommand{\klammern}[4][]%
{\ifthenelse{\equal{#1}{}}{\left#2}{\csname#1\endcsname#2}%
#4\ifthenelse{\equal{#1}{}}{\right#3}{\csname#1\endcsname#3}}

\newcommand{\Mod}[1]{ ({\mathrm{mod}\ #1})}

\renewcommand{\epsilon}{\DONOTUSE}  

\renewcommand{\vartheta}{\DONOTUSE} 

\date{\today}

\allowdisplaybreaks

\begin{document}

\title{Stieltjes constants for  Hurwitz-type $L$-series \\ attached to periodic arithmetic functions}

\author{Abdelmejid Bayad}
\address{ Laboratoire de Math\'ematiques et Mod\'elisation d'\'Evry, France.}
\email{abdelmejid.bayad@univ-evry.fr}

\author{Salah Eddine Rihane}
\address{National Higher School of Mathematics, Scientific and Technology Hub of Sidi Abdellah, P.O. Box 75, Algiers 16093, Algeria.}
\email{salahrihane@hotmail.fr}




\begin{abstract}
We investigate generalized Stieltjes constants attached to Hurwitz-type Dirichlet series of the form
$$
L(s, f, x)=\sum_{n \geq 0} \frac{f(n)}{(n+x)^s}
$$
where $f$ is a periodic arithmetic function and $x\in\mathbb C\backslash\{0,-1,-2,\dots\}.$ 
We prove that these series admit meromorphic continuation to the whole complex plane, with at
most a simple pole at  $s=1$. We further derive functional equations satisfied by $L(s,f,x).$
 Our main result  gives an explicit formulas of the Stieltjes constants $\gamma_n(f, x)$  appearing in the Laurent expansion at $s=1$.  In particular, we obtain explicit evaluations in the case of rational  $x$.   As an application, we derive explicit closed-form expressions for the first three generalized Stieltjes constants associated with nonprincipal real primitive odd
Dirichlet characters arising from quadratic fields. These expressions are given in terms of the class number of the corresponding quadratic number field. 
These results extend and unify recent work of Blagouchine and others on generalized Stieltjes constants  and Hurwitz-type $L$-functions.
\end{abstract}
\maketitle

\section{Introduction and motivation} \label{sec1} 
The Hurwitz zeta function is a classical special function of analytic number theory and complex analysis. For complex parameters  $x\not\in\{0,-1,-2,\cdots \}$  and  $\Re(s)>1$,  it is defined by the series  
$$\zeta(s,x)=\displaystyle\sum_{k=0}^{\infty} (x+k)^{-s}.$$  
This function admits a meromorphic continuation to the whole complex plane, with the exception of a simple pole at  $s=1$,  whose residue is equal to $1$. Consequently, near $s=1$ , the Hurwitz zeta function can be expanded into a Laurent series of the form

\begin{equation}\label{eq:Laurent-Hurwitz}
\zeta (s,x)=\dfrac{1}{s-1}+\sum_{n=0}^{\infty} (-1)^n \gamma_n(x) \dfrac{(s-1)^n}{n!}  . 
\end{equation}

The coefficients $\gamma_n(x)$  appearing in this expansion are known as the  Stieltjes constants.  They play an important role in the local behavior  of $\zeta(s,x)$ around its singularity and arise naturally in a variety of analytic and arithmetic contexts. When  $x=1$,  the Hurwitz zeta function reduces to the Riemann zeta function
$\zeta (s)=\displaystyle\sum_{k=1}^{\infty} k^{-s}=\zeta(s,1),$ and the corresponding constants  $\gamma_n(1)$ are usually denoted $\gamma_n$.   
In particular  the first constant $\gamma_0=\gamma\approx0.5772156649015328606\ldots$ is the celebrated Euler--Mascheroni constant, introduced by Euler \cite{Euler:1736} and later studied in depth  by Stieltjes  and  Hermite in their correspondance   \cite[vol. I, Lettre 71-- Lettre 75]{HS:1905}. The next constant  $\gamma_1\approx-0.072815845\ldots$
  was also computed by Stieltjes and has since attracted considerable attention due to its appearance in asymptotic formulas, summation methods, and the theory of special functions.   More generally, generalized Stieltjes constants provide a refined description of the singular structure of zeta functions and are connected to a wide range of topics, including the distribution of prime numbers, asymptotic expansions, and analytic regularization.
  
  Their study remains an active area of research, driven both by their intrinsic mathematical interest and by their central role in modern analytic number theory.

\subsection{Motivation}
Recently, in \cite{Blagouchine:2019},   Johansson and Blagouchine developed an efficient algorithm, based on complex integration, for the arbitrary-precision computation of the  Stieltjes constants $\gamma_n(x).$

Earlier, Berndt   \cite{Berndt:1972}, using  the Euler-Maclaurin summation formula,, showed that  
these constants admit the asymptotic representation
\begin{equation}\label{eq:Slt-Hurwitz}
\gamma_k(x)=\lim_{n\longrightarrow\infty} \left\{\sum_{m=0}^n \dfrac{\log^k (m+x)}{m+x}- \dfrac{\log^{k+1}(n+x)}{k+1}\right\}.
\end{equation}
 In \cite[Corollary 1.2]{Knopfmacher:1978}, Knopfmacher introduced
and investigated the constants
\begin{equation}\label{eq:Slt-Arithmetic-progressions}
\gamma_k(r,q)=\lim_{t\longrightarrow\infty} \left\{\sum_{n\leq t \atop n\equiv r \mod q} \dfrac{\log^k n}{n}- \dfrac{\log^{k+1}t}{q(k+1)}\right\}
\end{equation}
 which may be regarded as the $k$-th generalized Euler constants modulo $q$ associated with the arithmetic progression  $r\pmod q.$\\
 
Briggs and Buschman \cite{Briggs-Buschman:1961}, as well as Briggs \cite{Briggs:1962}, studied these constants using the Mellin transform in a direct manner and thereby established the convergence of the limit defining  $\gamma_k(r,q).$ 
 Knopfmacher, on the other hand, developed an approach based on the Laplace transform. 
 This method applies in a more general setting and not only yields 
 the convergence of the same limit, but also provides an explicit error estimate;  
  \cite[Theorem 2.1, Theorem 2.2]{Knopfmacher:1978}.   Subsequently, Dilcher \cite{Dilcher:1992} and Kamenitsu \cite{Kanemitsu:1989}  extended this line of work by examining higher-order generalized Euler--Stieltjes constants through sums over arithmetic progressions.
   Shirasaka  \cite{Shirasaka:2002} studied   the partial zeta-function

$$
\zeta(s,r , q)=\sum_{\substack{h=1 \\ h \equiv r(q)}}^{\infty} \frac{1}{n^s}, 
$$

and proved that it admits the Laurent expansion

$$
\zeta(s, r, q)=\frac{1}{q(s-1)}+\sum_{k=0}^{\infty} \frac{(-1)^k}{k!} \gamma_k(r, q)(s-1)^k .
$$
Moreover, the constants $ \gamma_k(r, q)$ can be expressed in terms of the Stieltjes constants attached to rational arguments as
$$
\gamma_k(r, q)=\frac{1}{q} \sum_{j=0}^k\binom{k}{j}(\log q)^{k-j} \gamma_j\left(\frac{r}{q}\right)-\frac{\log ^{k+1} q}{q(k+1)}.
$$
A further generalization was proposed by Chatterjee  \cite{Chatterjee:2019}, who introduced,for  $\alpha \in(0,1)$,  a family of constants  defined by


$$
\zeta_k(\alpha, r, q):=\lim _{t \rightarrow \infty}\left(H_k(t, \alpha, r, q)-\frac{I_k(x, \alpha)}{q}\right),
$$
where

$$
H_k(t, \alpha, r, q):=\sum_{\substack{n \leq t \\ n \equiv r(\bmod q)}} \frac{\log ^k n}{n^\alpha}, \quad I_k(t, \alpha):=t^{1-\alpha} \sum_{i=0}^k(-1)^i \frac{k!}{(k-i)!} \frac{\log ^{k-i} t}{(1-\alpha)^{i+1}} .
$$

He showed that these constants arise as coefficients in the Taylor expansion of the partial zeta function

$$
\zeta(s ; r, q)=\sum_{n \geq 1 \atop  n \equiv r(\bmod q)} n^{-s}
$$
around the point $s=\alpha$. \\
\subsection{Goal of the paper}
Motivated by these developments, 
we extend the study of Euler-Stieltjes constants to a broader class of zeta-type functions, namely the Hurwitz--Lerch $L$-functions attached to arithmetic functions that are periodic with a strictly positive integer period $q$.


More precisely, for a periodic arithmetic function  $f$  of period $q$, we introduce, for  
 $ x \in \mathbb{C} \setminus \mathbb{Z}_{\leq 0}$,   the associated  Hurwitz-type $L$-functions 

\begin{equation}\label{def:L-function}
L(s,f,x)=\sum_{n\geq 0} \dfrac{f (n)}{(n+x)^s},\quad \operatorname{Re}(s)>1,
\end{equation}

and  the corresponding $L$-function 
\begin{equation}\label{def:L-function-2}
L(s,f)=\sum_{n\geq 1} \dfrac{f (n)}{n^s},\quad \operatorname{Re}(s)> 1
\end{equation} 
these functions are related by the identity 
$$ L(s,f)= 
L(s,f,q)+\displaystyle\sum_{0<n<q} \dfrac{f(n)}{n^s} 
$$
with the last term vanishing when $q=1$. \\

The function $L(s,f,x)$  includes, as special cases, several classical and historically important functions in analytic
number theory.
Both $L(s, f,x)$ and $L(s,f)$   admit Laurent expansions in a neighborhood of  $s=1$, namely


\begin{equation}\label{Laurent-2-a}
L(s,f,x)= \dfrac{\sum_{r=0}^{q-1} f(r)}{q(s-1)}+ \sum_{k=0}^{\infty} \dfrac{(-1)^k}{k!} \gamma_k(f,x) (s-1)^k,
\end{equation}
and
\begin{equation}\label{Laurent-2-2-a}
L(s,f)= \dfrac{\sum_{r=0}^{q-1} f(r)}{q(s-1)}+ \sum_{k=0}^{\infty} \dfrac{(-1)^k}{k!} \gamma_k(f) (s-1)^k
\end{equation}
where $\gamma_k(f,x)$ and $\gamma_k(f)$  denote the generalized Stieltjes constants associated with  $L(s,f,x)$ and $L(s,f)$, respectively.\\


 To highlight the scope of our main results, we provide the  following explicit formulas for   $\gamma_k(f, x)$  and a concrete algebraic application.
\begin{theorem}
For every integer $k \geq 0$, we obtain 

$$
\gamma_k(f, x)=\sum_{r=0}^{q-1} f(r)\left(-\frac{\log ^{k+1} q}{q(k+1)}+\frac{1}{q} \sum_{j=0}^k\binom{k}{j} \gamma_j\left(\frac{x+r}{q}\right) \log ^{k-j} q\right) .
$$

Assume that  $\displaystyle\sum_{r=1}^q f(r)=0$, then 
\begin{equation}
L^{(k)}(1,x,f) = (-1)^k \gamma_k(f,x).
\end{equation}
Simplified  formulas  for $\gamma_0(f,x)$, $\gamma_1(f,x)$, and $\gamma_2(f,x)$, valid for rational numbers $0<x\leq q$, are obtained, see Theorem \ref{thm:gamma_0(f,x)}, Theorem \ref{thm:gamma_1(f,x)},  and Theorem \ref{thm:gamma_2(f,x)}. 
\end{theorem}

 For instance, this work yields explicit formulas in the following setting.
\begin{theorem}
Let $\chi$ be a nonprincipal, real, primitive odd character modulus $q\geq 7$. Denote by $h(-q)$ the class number of the imaginary quadratic field $\mathbb{Q}(\sqrt{-q})$. 
Then $\gamma_0(\chi)$, $\gamma_1(\chi)$ and $\gamma_2(\chi)$ can be expressed in terms of $h(-q)$
\begin{eqnarray}
&\gamma_0(\chi)=&\dfrac{\pi}{\sqrt{q}} h(-q),\\
&\gamma_1(\chi)=& -\dfrac{\pi}{\sqrt{q}} (\gamma+\log (2\pi)) h(-q)+ \dfrac{\pi}{\sqrt{q}} \sum_{l=1}^{q-1}\chi (l) \log \Gamma\left( \dfrac{l}{q}\right)\\ 
&\gamma_2(\chi)  =&-\dfrac{\pi}{\sqrt{q}} \left(\log^2q-2\log q(\gamma+\log (2\pi q))+\dfrac{\pi^2}{12}+(\gamma+\log (2\pi q))^2\right)  h(-q)\\
   & &+\dfrac{2\pi}{\sqrt{q}} \left(\log q-\pi (\gamma+\log (2\pi q))\right)  \sum_{l=1}^{q-1}\chi (l) \log \Gamma\left( \frac{l}{q} \right)+\dfrac{\pi}{\sqrt{q}}  \sum_{l=1}^{q-1}\chi (l) \zeta''\left(0, \frac{l}{q}\right).\nonumber 
\end{eqnarray}
\end{theorem}


\subsection{Summary of the paper}
In section \ref{section 2} we explore the analytic continuation of the functions  $s\to L(s,f,x)$, and their  associated functional equations. Section \ref{section 4} presents the statements of our main results. 
 The detailed proofs of these results are provided in Section  \ref{section 5}. 
 Section \ref{section 6} is dedicated  to explicit computations  for $\gamma_0(f,x)$, $\gamma_1(f,x)$ and $\gamma_2(f,x)$ at  rational  arguments $x$.    Finally, Section \ref{section 7} offers estimates for 
  alternating Hurwitz zeta function and various Dirichlet $L$-functions associated to different modular characters.

\section{Analytic continuation and functional equation}\label{section 2} 
In the following, we clarify the analytic continuation of the functions $L(s,f,x)$ by  their connection to the Hurwitz zeta function. By introducing
a new zeta function $l(s,f^{\pm},x)$ through the Fourier transform of the periodic function $f$, we establish a functional equation that relates $L(s,f,x)$ and $l(s,f^{\pm},x).$

\subsection{Analytic continuation} Using the periodicity of $f$, the function $L(s, f,x)$ and $L(s,f)$ can be expressed in terms of the Hurwitz Zeta function as the following 
\begin{equation*}
L(s,f,x)= \sum_{r=0}^{q-1}\sum_{k=0}^\infty \dfrac{f(qk+r)}{(kq+r+x)^s}=q^{-s} \sum_{r=0}^{q-1}f(r)\sum_{k=0}^\infty \dfrac{1}{\left(k+\dfrac{x+r}{q}\right)^s}
,\ \Re(s)>1.
\end{equation*}
Then  
\begin{equation}\label{eq:L-Hurwitz1}
L(s,f,x) = q^{-s} \sum_{r=0}^{q-1}f(r) \zeta \left(s,\dfrac{x+r}{q}\right),
\end{equation}
Similarly   we obtain the reduction formula
\begin{equation}\label{eq:L-Hurwitz2}
L(s,f)= \sum_{r=1}^{q}\sum_{k=0}^\infty \dfrac{f(qk+r)}{(kq+r)^s}=q^{-s} \sum_{r=1}^{q}f(r)\sum_{k=0}^\infty \dfrac{1}{\left(k+\dfrac{r}{q}\right)^s} = q^{-s} \sum_{r=1}^{q}f(r) \zeta \left(s,\dfrac{r}{q}\right).
\end{equation}

Since $\zeta(s, x)$ admits an analytic continuation to the whole complex plane except for a simple pole at $s = 1$ with residue 1,  then the series $\zeta(s,x)-\frac1{s-1}$ has the 
abscissa of convergence $0$ and allows an analytic extension to the whole complex plane. It follows that $L(s, f,x)$ also extends analytically to $\mathbb{C}$, with a simple pole at $s = 1$ if and only if   $\displaystyle\sum_{r=0}^{q-1} f(r)\neq 0$.  
Then for  $\displaystyle\sum_{r=0}^{q-1} f(r)=0$,  $L(s, f,x)$ can be extended analytically to $\mathbb{C}$.

\subsection{Functional Equations }
Since $f$ is a $q$-periodic function, we can define the Fourier transformation of $f$ by
\begin{equation}
\widehat{f}(n)=\dfrac{1}{q} \sum_{r=1}^qf(r) e^{\frac{-2\pi irn}{q}},
\end{equation}
and hence we have the Fourier inversion formula
\begin{equation}\label{eq:f-versus-f^}
f(n)=\dfrac{1}{\sqrt{q}} \sum_{r=1}^{q} \widehat{f}(r)e^{\frac{2\pi i rn}{q}}.
\end{equation}
Next, for  any $x$  we  consider the functions
\[ l(s,f^{\pm},x)=\sum_{n=1}^\infty \dfrac{\displaystyle e^{\frac{\pm 2\pi in x}{q}}\widehat{f}(\pm n)}{n^s}.\]

For $\operatorname{Re}(s)>1$ and $0<x\leq 1$, the Hurwitz zeta function satisfies the following identity
\begin{equation}\label{eq:Functional-Hurwitz}
\zeta (1-s,x)=\dfrac{\Gamma(s)}{(2\pi)^s} \left( e^{-\pi is/2}\sum_{n=1}^\infty \dfrac{e^{2\pi in x}}{n^s}+ e^{\pi is/2}\sum_{n=1}^\infty \dfrac{e^{-2\pi in x}}{n^s}\right).
\end{equation}
By \eqref{eq:L-Hurwitz1} and \eqref{eq:Functional-Hurwitz}, for $0<x\leq 1$, we get
\begin{eqnarray*}
L(1-s,f,x) &=& \dfrac{\Gamma(s)}{(2q\pi)^s}\sum_{r=0}^{q-1}\left( f(r) \left( e^{-\pi is/2}\sum_{n=1}^\infty \dfrac{e^{\frac{2\pi in(x+r)}{q}}}{n^s}+ e^{\pi is/2}\sum_{n=1}^\infty \dfrac{e^{\frac{-2\pi in(x+r)}{q}}}{n^s}\right) \right)\\
   &=& \dfrac{\Gamma(s)}{(2q\pi)^s} \left(e^{-\pi is/2}\sum_{n=1}^\infty \dfrac{\displaystyle e^{\frac{2\pi in x }{q}}\sum_{r=0}^{q-1}f(r)e^{\frac{2\pi inr}{q}}}{n^s}+ e^{\pi is/2}\sum_{n=1}^\infty \dfrac{e^{\frac{-2\pi inx}{q}}\displaystyle\sum_{r=0}^{q-1}f(r)e^{\frac{-2\pi inr}{q}}}{n^s} \right)\\
   &=& \dfrac{\Gamma(s)}{(2\pi)^sq^{s-\frac{1}{2}}} \left(e^{-\pi is/2}\sum_{n=1}^\infty \dfrac{\displaystyle e^{\frac{2\pi in x}{q}}\widehat{f}(n)}{n^s}+ e^{\pi is/2}\sum_{n=1}^\infty \dfrac{e^{\frac{-2\pi in x}{q}}\widehat{f}(-n)}{n^s} \right)  
   \end{eqnarray*}
we thus obtain
\begin{eqnarray*}
L(1-s,f,x) = \dfrac{\Gamma(s)}{(2\pi)^sq^{s-\frac{1}{2}}} \left(e^{-\pi is/2} l(s,f^+,x)+e^{\pi is/2} l(s,f^-,x)\right).   
\end{eqnarray*}

Therefore we get the following functional equation
\begin{theorem}
 For  $\Re(s)>1$ and  $0<x\leq 1$, we have 
\begin{equation}
L(1-s,f,x)=
\dfrac{\Gamma(s)}{(2\pi)^sq^{s-\frac{1}{2}}} \left(e^{-\pi is/2} l(s,f^+,x)+e^{\pi is/2} l(s,f^-,x)\right).
\end{equation}
\end{theorem}
This result generalizes \cite[Theorem 2.1]{Lang:1976} or \cite[Lemma 2.1]{Murty-Pathak:2018}.

 In fact, when $f=\chi$ is a primitive Dirichlet character with conductor $q$, then  $$
\widehat{\chi}(n)=\overline{\chi}(n) S(\chi) \chi (-1),
$$
where $S(\chi):=\displaystyle\sum _{a=1}^{q-1} \chi (a) ^{2\pi ia/q}\neq 0$ is the Gauss sum. Therefore, we obtain that
$$
l(s,f^+, x)=\chi(-1)S(\chi) \sum_{n=1}^\infty \dfrac{e^{\frac{2\pi i xn}{q}}\overline{\chi}(n)}{n^s}\quad \text{and}\quad l(s,f^-, x)=\chi(-1)S(\chi) \sum_{n=1}^\infty \dfrac{e^{\frac{-2\pi i xn}{q}}\overline{\chi}(-n)}{n^s}=S(\chi) \sum_{n=1}^\infty \dfrac{e^{\frac{-2\pi i xn}{q}}\overline{\chi}(n)}{n^s}. 
$$
Hence, we get the corollary
\begin{cor} For any  $\chi$ be a primitive Dirichlet character with conductor $q$ and $\Re(s)>1$, we have
\begin{equation}
L(1-s,x,\chi)=
\dfrac{\Gamma(s) S(\chi)}{(2\pi)^sq^{s-\frac{1}{2}}} \left(\chi(-1) e^{-\pi is/2} l(s,\overline{\chi},x)+e^{\pi is/2} l(s,\overline{\chi},-x)\right).
\end{equation}
\end{cor}

\section{Main results}\label{section 4} 
In this section we state our  main results and give some applications.
\subsection{On the Laurent series of $L(s, f,x)$}

The functions $L(s, f,x)$ and $L(s,f)$  can be  expanded in the Laurent series in a neighborhood of $s=1$ in the following way 

\begin{equation}\label{Laurent-2-b}
L(s,f,x)= \dfrac{\sum_{r=0}^{q-1} f(r)}{q(s-1)}+ \sum_{k=0}^{\infty} \dfrac{(-1)^k}{k!} \gamma_k(f,x) (s-1)^k,
\end{equation}
and
\begin{equation}\label{Laurent-2-2-b}
L(s,f)= \dfrac{\sum_{r=0}^{q-1} f(r)}{q(s-1)}+ \sum_{k=0}^{\infty} \dfrac{(-1)^k}{k!} \gamma_k(f) (s-1)^k
\end{equation}
where $\gamma_k(f,x)$ and $\gamma_k(f)$  denote the Stieltjes constants for  $L(s,f,x)$ and $L(s,f)$ respectively.

For \( q \geq 2 \), the function \( L(s, f, x) \) satisfies the identity  
\[
L(s,f,x+q) = L(s, x,f) - \sum_{n=0}^{q-1} \frac{f(n)}{(n + x)^s}.
\]

This allows us to derive the following recurrence relation for \( \gamma_k(f, x) \).
\begin{theorem}
\begin{equation}
\gamma_k(f, x+q) = \gamma_k(f, x) - \sum_{n=0}^{q-1} \frac{f(n)\log^k(n + x)}{n + x}, \quad \text{for } k = 0,1, 2, \ldots.
\end{equation}
\end{theorem}
Hence, it suffices to consider the case \( 0 < x \leq q \) in our study. Furthermore, 
we have
$$
L(s,f,q)=L(s,f)-\sum_{0<n<q} \dfrac{f(n)}{n^s}.
$$
As a consequence, we deduce that
\begin{proposition}
\begin{equation}\label{eq:gamma_k(f)-vesus-gamma-k(f-x)}
\gamma_k(f) = \gamma_k(f, q) + \sum_{0<n<q} \frac{f(n)\log^kn}{n}, \quad \text{for } k = 0,1, 2, \ldots.
\end{equation}
\end{proposition}

Our second main result provides an explicit decomposition of the coefficients  $\gamma_k(f,x)$ in terms of arithmetic progressions modulo $q.$  Throughout this section, we assume that  $0<x\leq q..$

\begin{theorem}\label{thm:01}
For  every integer $k\geq 0$,  we have
\begin{equation}
\gamma_k(f,x)=\sum_{r=0}^{q-1} f (r) \gamma_k(r,q,x),
\end{equation}
where
\begin{equation}
\gamma_k(r,q,x):=\lim_{t\longrightarrow \infty} \sum_{n\leq t \atop n\equiv r \mod q} \left( \dfrac{\log^k(n+x)}{n+x} - \dfrac{\log^{k+1}(t+x)}{q(k+1)}\right).
\end{equation}
\end{theorem}
We next derive an alternative closed form for these coefficients.

\begin{theorem}\label{thm:02} 

Let $q\geq 1$ and $0\leq r\leq q-1.$ Then, for every integer  $k\geq 0$,

\begin{equation}\label{gamm_k(r,q,x)-versus-gamma_k}
\gamma_k(r,q,x)=-\dfrac{\log^{k+1} q}{q(k+1)}+ \dfrac{1}{q} \sum_{j=0}^k \binom{k}{j} \gamma_j\left(\dfrac{x+r}{q}\right) \log^{k-j}q.
\end{equation}
\end{theorem}
In particular, we obtain:
\begin{equation}\label{gamma_0(r,q,x)-versus-gamma}
\gamma_0(r,q,x) =-\dfrac{\log q}{q}-\dfrac{1}{q} \psi \left(\dfrac{x+r}{q}\right),
\end{equation}
\begin{equation}\label{gamma_1(r,q,x)-versus-gamma}
\gamma_1(r,q,x) = -\dfrac{\log^2 q}{2q} -\dfrac{\log q}{q} \psi \left(\dfrac{x+r}{q}\right)+ \dfrac{1}{q} \gamma_1\left(\dfrac{x+r}{q}\right),
\end{equation}
and 
\begin{eqnarray}\label{gamma_2(r,q,x)-versus-gamma}
\gamma_2(r,q,x)   &=& -\dfrac{\log^3 q}{3q} -\dfrac{\log^2 q}{q} \psi \left(\dfrac{x+r}{q}\right)+  \dfrac{2\log q}{q} \gamma_1\left(\dfrac{x+r}{q}\right)+ \dfrac{1}{q} \gamma_2\left(\dfrac{x+r}{q}\right).\nonumber
\end{eqnarray}
We now generalize Proposition 3.2 of Knopfmacher \cite{Knopfmacher:1978} .

\begin{lemma}\label{lemma3.7}
For any integer $k\geq 0$,  the series $$\displaystyle\sum_{n=1}^\infty \dfrac{f(n) \log^k n}{n+x}$$
converges, and its sum equals  $\gamma_k(f,x)$, if and only if $$\displaystyle\sum_{r=1}^q f(r)=0.$$
\end{lemma}
\begin{proof} 
This follows directly from the identity
$$
\sum_{n=1}^\infty \dfrac{f(n) \log^k (n+x)}{n+x} = \sum_{r=1}^q f(r) \left( \sum_{n\leq t \atop n\equiv r \mod q} \dfrac{\log^k(n+x)}{n+x} - \dfrac{\log^{k+1}(t+x)}{q(k+1)}\right)+\dfrac{\log^{k+1}(t+x)}{q(k+1)} \sum_{r=1}^q f(r),
$$
and the last term vanishes precisely when  $\displaystyle\sum_{r=1}^q f(r)=0.$
\end{proof}
\subsection{Application to derivatives of  $L(s,f,x)$ at $s=1$}
Using Lemma 3.5 together with Theorem 3.4, we obtain the following important consequence.
\begin{theorem}[Derivatives values of $L$-functions at $s=1$]\label{thm:03}
Assume that  $\displaystyle\sum_{r=1}^q f(r)=0$. Then, for every  integer $k\geq 0$,
\begin{equation}
L^{(k)}(1,x,f) = (-1)^k \gamma_k(f,x).
\end{equation}
\end{theorem}
This extends Lemma 2.4  of  \cite{Murty-Pathak:2018}.
\subsection{ Generating function of  $\gamma_k(f,x)$}
Finally, we give the exponential generating function of the coefficients
$\gamma_k(f,x)$.

\begin{theorem}\label{thm:05}
We have
\begin{equation}
\sum_{k=0}^\infty \dfrac{\gamma_k(f,x)}{k!}= \sum_{r=0}^{q-1} f(r)\left[-\dfrac{1}{q}+ \left(\dfrac{1}{2}-\dfrac{x+r}{q}\right) \log q +\dfrac{1}{2} \log (2\pi)-\log \Gamma \left(\dfrac{x+r}{q}\right) \right]
\end{equation}
\end{theorem}
\subsection{On some properties of $\gamma_k(r,q,x)$}
In this section, we study several properties of $\gamma_k(r,q, x)$. Note that, 
\begin{equation}\label{eq:per}
\gamma_k(r\pm q,q,x)=\gamma_k(r,q,x).
\end{equation}
So we may restrict our attention to $1\leq r\leq q$. 
\begin{theorem}\label{thm:06}
For any integer $k\geq 0$,  we have

\begin{equation}\label{gamma_sum_k}
\sum_{r=0}^{q-1} \gamma_k(r,q, x) = \gamma_k (x).
\end{equation}
\end{theorem}
\begin{theorem}\label{thm:07}
Let $k\geq 0$ and $q\geq 1$ and $d=\gcd(r,q)$. Then  
\begin{equation}\label{eq:dev-gamma}
\gamma_k(r,q,x)=\dfrac{1}{d} \sum_{j=0}^k \binom{k}{j}\log^j d \cdot \gamma_{k-j} \left(\dfrac{r}{d},\dfrac{q}{d},\dfrac{x}{d}\right)- \dfrac{\log^{k+1}d}{q(k+1)}.
\end{equation}
\end{theorem}

\begin{theorem}\label{thm:08}
For all $k\geq 0$ and $r\geq 1$ we have
\begin{equation}\label{eq:special-1}
\gamma_k(r,r,r) =\dfrac{1}{r} \sum_{j=0}^k \binom{k}{j} \gamma_j \log^{k-j}r - \dfrac{\log^{k+1}r}{r(k+1)}.
\end{equation}
\end{theorem}

\begin{theorem}\label{thm:09}
For all integers $k\geq 0,r\geq 1$ and $m\ge 2$ we have
\begin{equation}\label{eq:special-2}
\gamma_k(r,r,mr) =\dfrac{1}{r} \sum_{j=0}^k \binom{k}{j} \left(\gamma_j -\sum_{n=1}^{m-1}\dfrac{\log^k n}{n}\right) \log^{k-j}r- \dfrac{\log^{k+1}r}{r(k+1)}.
\end{equation}
\end{theorem}
\begin{theorem}\label{thm:10}

For $k\geq0$, $q\geq 2$, and $0\leq r\leq q-1$ we have
\begin{equation}
 q\gamma_k(r,q,x)=\gamma_k(x)+ \sum_{n=1}^\infty \dfrac{\log^k(n+x)}{n+x} g_r^{(q)} (n),
\end{equation}
where $g_r^{(q)}(n)$ be the arithmetic $q$-periodic function defined by
$$
g_r^{(q)}(n)= \left\{
\begin{array}{ll}
q-1, & \text{when } n\equiv r \pmod q,\\
-1,  & \text{otherwise}.
\end{array}
\right.
$$ 
\end{theorem}
As an application of $\gamma_k(r,q,x)$ we give the following results, which generalized Dilcher  \cite[Proposition 8]{Dilcher:1992} and 
Lehmer \cite[Theorem 9]{Lehmer:1975} results .
\begin{theorem}\label{thm:04}
Let $k\geq 0$, $N\geq 2$, and let $(r_1,m_1),(r_2,m_2),\ldots, (r_N,m_N)$ be pairs of positive integers for which $1\leq r_j\leq m_j$, $x_j\in\mathbb{C}$, $j=1,2,\ldots, N$, and such that the rational numbers $r_j/m_j$ are distinct. Then the series
\begin{equation}\label{def:S_k}
S_k:=\sum_{n=0}^\infty \left\{ c_1 \dfrac{\log^k(m_1n+r_1+x_1)}{m_1n+r_1+x_1}+ \cdots + c_N \dfrac{\log^k(m_Nn+r_N+x_N)}{m_Nn+r_N+x_N}\right\} 
\end{equation}
converges if and only if 
\begin{equation}
\sum_{j=1}^N \dfrac{c_j}{m_j} \log^\nu m_j=0\quad \text{for } \nu=0,1,\ldots,k.
\end{equation}
If the series \eqref{def:S_k} converges, then
\begin{equation}
S_k=\sum_{j=1}^N c_j \left( \gamma_k(r_j,m_j,x_j) +  \dfrac{\log^{k+1} m_j}{m_j(k+1)}\right).
\end{equation}
\end{theorem}
\section{Proof of main results}\label{section 5} 
In this section we outline the proofs of our main results.
\subsection{Proof of Theorem \ref{thm:01}}
Taking 
$$
f(n)= 
\begin{cases}
1, & \text{if } n\equiv r \mod q,\\
0, & \text{otherwise,}
\end{cases}
$$
replace it, in \eqref{eq:L-Hurwitz1}, we get that
$$
\sum_{n=0\atop n\equiv r\mod q}^\infty \dfrac{1}{(n+x)^s}= q^{-s} \zeta\left(s,\dfrac{x+r}{q}\right).
$$
Therefore, we deduce that
\begin{equation}
L(s,f,x)= \sum_{r=0}^{q-1} f(r) \sum_{n=0\atop n\equiv r\mod q}^{\infty} \dfrac{1}{(n+x)^s}.
\end{equation}
The Taylor series of $\displaystyle\sum_{0\leq n\leq t\atop n\equiv r\mod q} \dfrac{1}{(n+x)^s}$ around $s=1$ is given by
\begin{eqnarray*}
\sum_{0\leq n\leq t\atop n\equiv r\mod q} \dfrac{1}{(n+x)^s} &=& \sum_{k=0}^{\infty} \dfrac{(-1)^k}{k!} (s-1)^k \sum_{n\leq t\atop n\equiv r\mod q} \dfrac{\log^k(n+x)}{n+x}\\
   &=& \sum_{k=0}^{\infty} \dfrac{(-1)^k}{k!} (s-1)^k \left(\sum_{n\leq t\atop n\equiv r\mod q} \dfrac{\log^k(n+x)}{n+x}-\dfrac{\log^{k+1}(t+x)}{q(k+1)}\right)+ \sum_{k=0}^{\infty} \dfrac{(-1)^k}{k!} (s-1)^k \dfrac{\log^{k+1}(t+x)}{q(k+1)}\\
   &=& \sum_{k=0}^{\infty} \dfrac{(-1)^k}{k!} (s-1)^k \left(\sum_{n\leq t\atop n\equiv r\mod q} \dfrac{\log^k(n+x)}{n+x}-\dfrac{\log^{k+1}(t+x)}{q(k+1)}\right)\\\
   & &-\dfrac{1}{q(s-1)} \sum_{k=0}^{\infty} \dfrac{(-1)^{k+1}}{(k+1)!} \log^{k+1}(t+x) (s-1)^{k+1}\\
   &=& \sum_{k=0}^{\infty} \dfrac{(-1)^k}{k!} (s-1)^k \left(\sum_{n\leq t\atop n\equiv r\mod q} \dfrac{\log^k(n+x)}{n+x}-\dfrac{\log^{k+1}(t+x)}{q(k+1)}\right)\\\
   & &-\dfrac{1}{q(s-1)} \left(\sum_{k=0}^{\infty} \dfrac{(-1)^{k}}{k!} \log^{k}(t+x) (s-1)^{k}-1\right)\\
   &=& \sum_{k=0}^{\infty} \dfrac{(-1)^k}{k!} (s-1)^k \left(\sum_{n\leq t\atop n\equiv r\mod q} \dfrac{\log^k(n+x)}{n+x}-\dfrac{\log^{k+1}(t+x)}{q(k+1)}\right)-\dfrac{1}{q(s-1)}\left(\dfrac{1}{(x+t)^{s-1}}-1\right).
\end{eqnarray*}
Using the fact that $\displaystyle\lim_{t\longrightarrow \infty}\left(\dfrac{1}{(x+t)^{s-1}}-1\right)=-1$ for $\operatorname{Re}(s)>1$, we obtain
$$
\sum_{n=0\atop n\equiv r\mod q}^\infty \dfrac{1}{(n+x)^s} = \dfrac{1}{q(s-1)}+\sum_{k=0}^{\infty} \dfrac{(-1)^k}{k!} \lim_{t\longrightarrow\infty}\left(\sum_{n\leq t\atop n\equiv r\mod q} \dfrac{\log^k(n+x)}{n+x}-\dfrac{\log^{k+1}(t+x)}{q(k+1)}\right)(s-1)^k.
$$
Hence, we obtain our Theorem \ref{thm:01}.
\subsection{Proof of Theorem \ref{thm:02}}
Using \eqref{eq:L-Hurwitz1} and \eqref{eq:Laurent-Hurwitz} we have the following expansions of $L(s,f,x)$ 
\begin{equation}\label{eq:Laurent-L}
L(s,f,x)=q^{-s}\dfrac{\sum_{r=0}^{q-1} f(r)}{s-1}+q^{-s} \sum_{k=0}^\infty \dfrac{(-1)^k}{k!} \left\{\sum_{r=0}^{q-1} f(r)\gamma_k\left(\dfrac{x+r}{q}\right)\right\} (s-1)^k,
\end{equation}
Now, using the fact that
$$
q^{-s}=\sum_{k=0}^\infty \dfrac{(-1)^k}{k!}\cdot \dfrac{\log^k q}{q} (s-1)^k= \dfrac{1}{q}\left( 1+ \sum_{k=1}^\infty \dfrac{(-1)^k}{k!}\log^k q (s-1)^k\right)= \dfrac{1}{q}\left( 1+\sum_{k=0}^\infty \dfrac{(-1)^{k+1}}{(k+1)!} \log^{k+1} q (s-1)^{k+1}\right)
$$
we obtain
\begin{equation}\label{eq2}
q^{-s}\dfrac{\sum_{r=0}^{q-1} f(r)}{(s-1)}= \dfrac{\sum_{r=0}^{q-1} f(r)}{q(s-1)}- \dfrac{\sum_{r=0}^{q-1} f(r)}{q} \sum_{k=0}^\infty \dfrac{(-1)^k}{(k+1)!} \log^{k+1} q (s-1)^k
\end{equation}
and
\begin{eqnarray}\label{eq3}
q^{-s} \sum_{k=0}^\infty \dfrac{(-1)^k}{k!} \left\{\sum_{r=0}^{q-1} f(r)\gamma_k\left(\dfrac{x+r}{q}\right)\right\} (s-1)^k &=& \left(\sum_{k=0}^\infty \dfrac{(-1)^k}{k!}\cdot \dfrac{\log^k q}{q} (s-1)^k\right) \left(\sum_{k=0}^\infty \dfrac{(-1)^k}{k!} \left\{\sum_{r=0}^{q-1} f(r)\gamma_k\left(\dfrac{x+r}{q}\right)\right\} (s-1)^k\right)\nonumber \\
   &=&\sum_{k=0}^{\infty} \left(\sum_{j=0}^k  \dfrac{(-1)^j\left\{\displaystyle\sum_{r=0}^{q-1} f(r)\gamma_j\left(\dfrac{x+r}{q}\right)\right\} }{j!}\cdot  \dfrac{(-1)^{k-j} \log^{k-j}q}{(k-j)!q} \right) (s-1)^k\nonumber\\
   &=&\dfrac{1}{q} \sum_{k=0}^{\infty} \dfrac{(-1)^k}{k!} \left(\sum_{r=0}^{q-1} f(r)\left\{\sum_{j=0}^k \binom{k}{j} \gamma_j\left(\dfrac{x+r}{q}\right)\log^{k-j}q\right\}\right) (s-1)^k.
\end{eqnarray}
We get the desired result when we combine relations  \eqref{eq:Laurent-L}, \eqref{eq2} and \eqref{eq3}.
\subsection{Proof of Theorem \ref{thm:05}}
Differentiating \eqref{eq:L-Hurwitz1} and \eqref{Laurent-2-b} with respect to $s$ we obtain
$$
-\dfrac{\sum_{r=0}^{q-1} f(r)}{q(s-1)^2}+\sum_{k=0}^\infty \dfrac{(-1)^k}{k!} \gamma_k(f,x) (s-1)^{k-1}= \sum_{r=0}^{q-1}\left\{-\ln q\; \zeta\left( s,\dfrac{x+r}{q}\right) +\zeta'\left( s,\dfrac{x+r}{q}\right) \right\} q^{-s} f(r)
$$
Letting $s=0$ and using Lerch's identities
$$
\zeta(0,a)=\dfrac{1}{2}-a \quad\text{and}\quad \zeta'(0,a)=\log \Gamma (a)- \dfrac{1}{2}\log (2\pi)
$$
we get the desired result.
\subsection{Proof of Theorem \ref{thm:06}}
For $k=0$ we have 

\begin{align*}
 \sum_{r=0}^{q-1} \gamma_0(r,q, x) =\sum_{r=0}^{q-1} \left[ -\dfrac{\log q}{q} - \frac{1}{q} \psi\left(\dfrac{x+r}{q}\right) \right] =  -\log q - \frac{1}{q} \sum_{r=0}^{q-1} \psi\left(\dfrac{x+r}{q}\right).
\end{align*}
Thanks to the multiplication formula for $\psi (x)$, we obtain 
\begin{align*}
\sum_{r=0}^{q-1} \gamma_0(r,q, x) =-\psi(x).
\end{align*}

For $k\geq 1$, we have
\begin{align*}
\sum_{r=0}^{q-1} \gamma_k(r,q, x) &= \sum_{r=0}^{q-1} \left[ -\frac{\log^{k+1}q}{q(k+1)} + \frac{1}{q} \sum_{j=0}^k \binom{k}{j} \gamma_j\left(\frac{x}{q}+\frac{r}{q}\right) \log^{k-j} q \right] \\
&= -\frac{\log^{k+1} q}{k+1} + \frac{1}{q} \sum_{r=0}^{q-1} \sum_{j=0}^k \binom{k}{j} \gamma_j\left(\frac{x}{q}+\frac{r}{q}\right) \log^{k-j} q \\
&= -\frac{\log^{k+1} q}{k+1} + \frac{1}{q} \sum_{j=0}^k \binom{k}{j} \log^{k-j} q \sum_{r=0}^{q-1} \gamma_j\left(\frac{x}{q}+\frac{r}{q}\right) \\
&= -\frac{\log^{k+1} q}{k+1} +\dfrac{\log^kq}{q} \sum_{r=0}^{q-1} \gamma_0\left(\frac{x}{q}+\frac{r}{q}\right) +\frac{1}{q} \sum_{j=1}^k \binom{k}{j} \log^{k-j} q \sum_{r=0}^{q-1} \gamma_j\left(\frac{x}{q}+\frac{r}{q}\right) \\
&= -\frac{\log^{k+1} q}{k+1}- \dfrac{\log^kq}{q} \sum_{r=0}^{q-1} \psi\left(\frac{x}{q}+\frac{r}{q}\right) + \sum_{j=1}^k \binom{k}{j} \log^{k-j} q \left\{(-1)^j\left(  \frac{\log q}{j+1} - \psi(x) \right)\log^j q + \sum_{\ell=0}^{j-1} (-1)^\ell \binom{j}{\ell} \gamma_{j-\ell}(x) \log^\ell q\right\} 
\end{align*}
Therefore, we get 
\begin{align*}
\sum_{r=0}^{q-1} \gamma_k(r,q, x)&= -\frac{\log^{k+1}q}{k+1}- \dfrac{\log^kq}{q} \sum_{r=0}^{q-1} \psi\left(\frac{x}{q}+\frac{r}{q}\right) - \psi(x) \log^k q \sum_{j=1}^k (-1)^j \binom{k}{j} + \log^{k+1} q  \sum_{j=1}^k  \frac{(-1)^j}{j+1}\binom{k}{j} \\
& + \sum_{j=1}^k \sum_{\ell=0}^{j-1} (-1)^\ell \binom{k}{j} \binom{j}{\ell} \gamma_{j-\ell}(x)\log^{k-j+\ell} q  \\
&= -\frac{\log^{k+1}q}{k+1}- \dfrac{\log^kq}{q} \sum_{r=0}^{q-1} \psi\left(\frac{x}{q}+\frac{r}{q}\right)+ \psi(x) \log^k q  + \log^{k+1} q  \sum_{j=1}^k  \frac{(-1)^j}{j+1} \binom{k}{j} + \sum_{j=1}^k \sum_{\ell=0}^{j-1} (-1)^\ell \binom{k}{j} \binom{j}{\ell} \gamma_{j-\ell}(x)\log^{k-j+\ell} q\\
&= - \dfrac{\log^kq}{q} \sum_{r=0}^{q-1} \psi\left(\frac{x}{q}+\frac{r}{q}\right)  + \psi(x) \log^k q-\log^{k+1}q+\sum_{j=1}^k \sum_{\ell=0}^{j-1} (-1)^\ell \binom{k}{j} \binom{j}{\ell} \gamma_{j-\ell}(x)\log^{k-j+\ell} q\\
&=\sum_{j=1}^k \sum_{\ell=0}^{j-1} (-1)^\ell \binom{k}{j} \binom{j}{\ell} \gamma_{j-\ell}(x)\log^{k-j+\ell} q.
\end{align*}
where we have used the fact that
$$
\sum_{j=0}^k  \frac{(-1)^j}{j+1} \binom{k}{j} =  \dfrac{1}{k+1} \sum_{j=0}^k (-1)^j \frac{(k+1)!}{(j+1)! (k+1-(j+1))!}=\dfrac{-1}{k+1}\sum_{j=0}^k (-1)^{j+1} \binom{k+1}{j+1}= \dfrac{-1}{k+1}\sum_{j=1}^{k+1} (-1)^{j} \binom{k+1}{j}= \dfrac{1}{k+1} 
$$
and
\begin{align*}
\sum_{j=0}^k \sum_{\ell=0}^{j-1} (-1)^\ell \binom{k}{j} \binom{j}{\ell} \gamma_{j-\ell}(x)\log^{k-j+\ell} q &= \sum_{j=1}^k \gamma_{j}(x)\log^{k-j} q\sum_{\ell=0}^{k-j} (-1)^\ell \binom{k}{\ell +j} \binom{j+\ell}{\ell}\\
 &= \sum_{j=1}^k \gamma_{j}(x)\log^{k-j} q\sum_{\ell=0}^{k-j} (-1)^\ell \binom{k}{\ell } \binom{k-\ell}{k-j-\ell}.
\end{align*}
The last sum equals to zero if $j\neq k$ and 1 when $j=k$.
\subsection{Proof of Theorem \ref{thm:07}}
Put $r=dr'$ and $q=dq'$. We have
\begin{eqnarray*}
\gamma_k(r,q,x) &=&\lim_{t\longrightarrow +\infty}\left\{ \sum_{n\leq t \atop n\equiv dr'\Mod{ dq'}} \dfrac{\log^k(n+x)}{n+x}-\dfrac{\log^{k+1}(t+x)}{dq'(k+1)}\right\}\\
  &=& \lim_{t\longrightarrow +\infty} \left\{\sum_{dn\leq t \atop n\equiv r'\Mod{ q'}}  \dfrac{\log^k(dn+x)}{dn+x}-\dfrac{\log^{k+1}(t+x)}{dq'(k+1)}\right\}\\
  &=& \dfrac{1}{d}\lim_{t\longrightarrow +\infty}\left\{\sum_{nd\leq t \atop n\equiv r'\Mod{q'}} \dfrac{\log^k\left( d\left(n+\dfrac{x}{d}\right)\right)}{n+\dfrac{x}{d}}-\dfrac{\log^{k+1}(t+x)}{q'(k+1)}\right\}.
\end{eqnarray*}
Using the facts that
$$
\log^k\left( d\left(n+\dfrac{x}{d}\right)\right) = \left(\log d+\log  \left(n+\dfrac{x}{d}\right)\right)^k = \sum_{j=0}^k \binom{k}{j} \log^j d \log^{k-j}\left(n+\dfrac{x}{d}\right)
$$
and
\begin{eqnarray*}
\log^{k+1}(t+x) &=& \log^{k+1}\left( d \cdot \dfrac{t+x}{d}\right)\\
  &=& \sum_{j=0}^{k+1} \binom{k+1}{j}\log^j d\log^{k-j+1} \left(\dfrac{t+x}{d}\right)\\
  &=& \log^{k+1} d +\sum_{j=0}^k \binom{k+1}{j} \log^j d\log^{k-j+1} \left(\dfrac{t+x}{d}\right)\\
  &=& \log^{k+1} d +(k+1)\sum_{j=0}^k \binom{k}{j} \dfrac{\log^j d\log^{k-j+1} \left(\dfrac{t+x}{d}\right)}{k+1-j},
\end{eqnarray*}
we obtain
\begin{eqnarray*}
\gamma_k(r,q,x) &=& \dfrac{1}{d}\lim_{t\longrightarrow +\infty}\left\{\sum_{nd\leq t \atop n\equiv r'\Mod{q'}} \sum_{j=0}^{k} \binom{k}{j} \dfrac{\log^j d\log^{k-j} \left(n+\dfrac{x}{d}\right)}{n+\dfrac{x}{d}}   - \dfrac{\log^{k+1} d}{q'(k+1)} -\sum_{j=0}^k \binom{k}{j} \dfrac{\log^j d\log^{k-j+1} \left(\dfrac{t+x}{d}\right)}{q'(k+1-j)}\right\}\\
  &=& \dfrac{1}{d}\lim_{t\longrightarrow +\infty}\left\{\sum_{j=0}^{k} \binom{k}{j}\log^j d\left(\sum_{nd\leq t \atop n\equiv r'\Mod{q'}}  \dfrac{\log^{k-j} \left(n+\dfrac{x}{d}\right)}{n+\dfrac{x}{d}}- \dfrac{\log^{k-j+1} \left(\dfrac{t+x}{d}\right)}{q'(k+1-j)}\right)   - \dfrac{\log^{k+1} d}{q'(k+1)}  \right\}\\
  &=& \dfrac{1}{d}\sum_{j=0}^{k} \binom{k}{j}\log^j d\left(\lim_{t\longrightarrow +\infty}\left\{\sum_{nd\leq t \atop n\equiv r'\Mod{q'}}  \dfrac{\log^{k-j} \left(n+\dfrac{x}{d}\right)}{n+\dfrac{x}{d}}- \dfrac{\log^{k-j+1} \left(\dfrac{t+x}{d}\right)}{q'(k+1-j)}\right\}\right)   - \dfrac{\log^{k+1} d}{q'd(k+1)} \\
  &=&\dfrac{1}{d}\sum_{j=0}^{k} \binom{k}{j}\log^j d\left(\lim_{t\longrightarrow +\infty}\left\{\sum_{n\leq t \atop n\equiv r'\Mod{q'}}  \dfrac{\log^{k-j} \left(n+\dfrac{x}{d}\right)}{n+\dfrac{x}{d}}- \dfrac{\log^{k-j+1} \left(t+\dfrac{x}{d}\right)}{q'(k+1-j)}\right\}\right)   - \dfrac{\log^{k+1} d}{q(k+1)}\\
  &=& \dfrac{1}{d}\sum_{j=0}^{k} \binom{k}{j}\log^j d \gamma_{k-j}\left(\dfrac{r}{d},\dfrac{q}{d}, \dfrac{x}{d}\right) - \dfrac{\log^{k+1} d}{q(k+1)}.  
\end{eqnarray*} Thus we prove the Theorem \ref{thm:07}.

\subsection{Proof of Theorem \ref{thm:08}}
We set $r=d=x=q$ in \eqref{eq:dev-gamma} we get
$$
\gamma_k(r,r,r)=\dfrac{1}{r} \sum_{j=0}^k \binom{k}{j}\log^j r \cdot \gamma_{k-j} \left(1,1,1\right)- \dfrac{\log^{k+1}r}{r(k+1)}= \dfrac{1}{r} \sum_{j=0}^k \binom{k}{j} \gamma_{j}\log^{k-j}r- \dfrac{\log^{k+1}r}{r(k+1)},
$$
where we have used the fact that
\begin{eqnarray*}
\gamma_k(1,1,1) &=& \lim_{ t\longrightarrow +\infty} \left\{ \sum_{0\leq n\leq t}\dfrac{\log^k (n+1)}{n+1}-\dfrac{\log^{k+1}(t+1)}{k+1}\right\}\\
                &=& \lim_{ t\longrightarrow +\infty} \left\{ \sum_{1\leq n\leq t+1}\dfrac{\log^k n}{n}-\dfrac{\log^{k+1}(t+1)}{k+1}\right\}\\
                &=& \lim_{ t\longrightarrow +\infty} \left\{ \sum_{1\leq n\leq t}\dfrac{\log^k n}{n}-\dfrac{\log^{k+1}t}{k+1}\right\}= \gamma_k.
\end{eqnarray*}
\subsection{Proof of Theorem \ref{thm:09}}
We set $d=q=r$ and $x=rm$ in \eqref{eq:dev-gamma} we get
$$
\gamma_k(r,r,mr)=\dfrac{1}{r} \sum_{j=0}^k \binom{k}{j}\log^jd \gamma_{k-j} \left(1,1,m\right)- \dfrac{\log^{k+1}r}{r(k+1)}=\dfrac{1}{r} \sum_{j=0}^k \binom{k}{j}\log^{k-j}r \gamma_{j} \left(1,1,m\right)- \dfrac{\log^{k+1}r}{r(k+1)}.
$$
Now, we  compute  $\gamma_j(1,1,m)$:
\begin{eqnarray*}
\gamma_j(1,1,m) &=& \lim_{ t\longrightarrow +\infty} \left\{ \sum_{0\leq n\leq t}\dfrac{\log^j (n+m)}{n+m}-\dfrac{\log^{j+1}(t+m)}{j+1}\right\}\\
                &=& \lim_{ t\longrightarrow +\infty} \left\{ \sum_{m\leq n\leq t+m}\dfrac{\log^j n}{n}-\dfrac{\log^{j+1}(t+m)}{j+1}\right\}\\
                &=& \lim_{ t\longrightarrow +\infty} \left\{ \sum_{1\leq n\leq t+m}\dfrac{\log^j n}{n}-\dfrac{\log^{j+1}(t+m)}{j+1}\right\}-\sum_{n=1}^{m-1}\dfrac{\log^j n}{n} \\
                &=& \lim_{ t\longrightarrow +\infty} \left\{ \sum_{1\leq n\leq t}\dfrac{\log^j n}{n}-\dfrac{\log^{j+1}t}{j+1}\right\}-\sum_{n=1}^{m-1}\dfrac{\log^j n}{n}\\
                &=& \gamma_j-\sum_{n=1}^{m-1}\dfrac{\log^j n}{n}.
\end{eqnarray*}
Therefore, we obtain $ \gamma_j(1,1,m) =\gamma_j-\sum_{n=1}^{m-1}\dfrac{\log^j n}{n}.$ Thus we finish the proof of Theorem \ref{thm:09}.
\subsection{Proof of Theorem \ref{thm:10}}
Consider the finite Fourier series generated by $\gamma_k(r,q,x)$ namely
\begin{equation}\label{def:sigma_j}
\sigma_j(q,x):=\sum_{a=0}^{q-1}\gamma_k(a,q,x) e^{\frac{2\pi i a j}{q}}. 
\end{equation}
Then according to \eqref{gamma_sum_k},  we have
\begin{equation}\label{sigma_0}
\sigma_0(q,x)= \gamma_k(x)
\end{equation}
for $j\neq 0$, we have
\begin{eqnarray}\label{sigma_j_not0}
\sigma_j (q,x)&=&  \displaystyle\sum_{a=0}^{q-1} \left( \lim_{t\longrightarrow\infty} \sum_{n\leq t \atop n\equiv a \pmod q} \dfrac{\log^k (n+x)}{n+x}-\dfrac{\log^{k+1}(t+x)}{q(k+1)}\right) e^{\frac{2\pi i a j}{q}}\vspace{2mm}\\
         &=& \displaystyle\lim_{t\longrightarrow\infty}\left( \sum_{a=0}^{q-1}  \sum_{n\leq t \atop n\equiv a \pmod q} \dfrac{\log^k (n+x)}{n+x} e^{\frac{2\pi i a j}{q}} -\dfrac{\log^{k+1}(t+x)}{q(k+1)} \sum_{a=0}^{q-1} e^{\frac{2\pi i a j}{q}}\right)\vspace{2mm}\nonumber\\
         &=& \displaystyle\sum_{n=1}^\infty \dfrac{\log^k(n+x)}{n+x} e^{\frac{2\pi i jn}{q}}.\nonumber
\end{eqnarray}
Multiplying both sides of \eqref{def:sigma_j} by $e^{-\frac{2\pi ijr}{q}}$ and summing over $j$ gives us
$$
\sum_{j=0}^{q-1} \sigma_j(q,x) e^{-\frac{2\pi ijr}{q}}= \sum_{j=0}^{q-1} \sum_{a=0}^{q-1} \gamma_k(a, q ,x) e^{\frac{2\pi ij(a-r)}{q}} = \sum_{a=0}^{q-1} \gamma_k(a,q,x) \sum_{j=0}^{q-1} e^{\frac{2\pi ij(a-r)}{q}} =q\gamma_k(r,q,x).
$$
Thus, using \eqref{sigma_0} and \eqref{sigma_j_not0} we obtain
\begin{equation}
q\gamma_k(r,q,x)= \gamma_k(x)+ \displaystyle\sum_{n=1}^\infty \dfrac{\log^k(n+x)}{n+x} \sum_{j=1}^{q-1} e^{\frac{2\pi i j(n-r)}{q}}.
\end{equation}
The result follows from the definition of $g_r^{(q)} (n)$.
\subsection{Proof of Theorem \ref{thm:04}}
We have
$$
\begin{array}{lcl}
S_k &=& \displaystyle\lim_{x\longrightarrow\infty} \sum_{j=1}^N c_j \sum_{0\leq n\leq t} \dfrac{\log^k (m_jn+r_j+x_j)}{m_jn+r_j+x_j}\\
    &=& \displaystyle\lim_{x\longrightarrow\infty} \sum_{j=1}^N c_j \sum_{0< n\leq m_jt \atop n\equiv r_j \pmod {m_j}} \dfrac{\log^k (n+x_j)}{n+x_j}\\
    &=& \displaystyle\sum_{j=1}^N c_j\lim_{x\longrightarrow\infty} \left\{  \sum_{0< n\leq m_jt \atop n\equiv r_j \pmod {m_j}} \dfrac{\log^k (n+x_j)}{n+x_j} - \dfrac{\log^{k+1}(m_jt+x_j)}{m_j(k+1)} \right\}+ \lim_{x\longrightarrow\infty} \sum_{j=1}^N \dfrac{c_j \log^{k+1}(m_jt+x_j)}{m_j(k+1)}\\
    &=& \displaystyle\sum_{j=1}^N c_j \gamma_k(r_j,m_j,x_j) + \lim_{x\longrightarrow\infty} \sum_{j=1}^N \dfrac{c_j}{m_j(k+1)} \left(\log m_j+ \log\left( t+\dfrac{x_j}{m_j} \right)\right)^{k+1}\\
    &=& \displaystyle\sum_{j=1}^N c_j \gamma_k(r_j,m_j,x_j) +  \sum_{j=1}^N \dfrac{c_j\log^{k+1} m_j}{m_j(k+1)} +\lim_{x\longrightarrow\infty} \left(\sum_{\nu=0}^{k} \dfrac{1}{k+1}\binom{k+1}{\nu} \left(\sum_{j=1}^N \dfrac{c_j}{m_j} \log^{\nu} m_j\right)  \log^{k+1-\nu}\left( t+\dfrac{x_j}{m_j} \right)\right).
\end{array}
$$
\section{An explicit formula for $\gamma_0(f,x)$, $\gamma_1(f,x)$ and $\gamma_2(f,x)$ at rational  arguments }\label{section 6}
Let $0<x\leq q$ be rational number. We will give some explicit formula for $\gamma_0(f,x)$, $\gamma_1(f,x)$ and $\gamma_2(f,x)$ for rational numbers $x$. Let $\Psi (z)$ be the digamma function. The digamma function satisfies the recurrence relation
\begin{equation}\label{recu-Psi}
\Psi (x+1)=\Psi (x)+\dfrac{1}{x} 
\end{equation}
and the reflection formula
\begin{equation}\label{ref-Psi}
\Psi(1-x)-\Psi(x)=\pi \cot \pi x.
\end{equation}
Let $r$ and $m$ be positive integers such that $m>1$ and  $r < m$. The digamma function may be expressed in terms of Euler's constant and a finite number of elementary functions \cite[page 584]{Blagouchine:2015}:
\begin{equation}\label{Psi-Rational}
\Psi\left(\dfrac{r}{m}\right)=-\gamma-\log (2m)-\dfrac{\pi}{2}\cot\left(\dfrac{r\pi}{m}\right)+\sum_{l=1}^{m-1} \cos\left(\dfrac{2\pi l r}{m}\right) \log\sin\left(\dfrac{\pi l}{m}\right)
\end{equation}
which holds, because of its recurrence equation, for all rational arguments.

The first generalized Stieltjes constant at rational argument may be evaluated in a quasi-closed form via the following formula \cite[Page 559]{Blagouchine:2015}:
\begin{align}\label{First-Stel-Rational-2}
\gamma_1\left( \frac{r}{m} \right) &= \gamma_1- \gamma \log (2 m)  -\log^2 (2) - \log 2\cdot \log (\pi m) - \frac{\log^2 m}{2} -\dfrac{\pi }{2}(\gamma + \log (2\pi m)) \cdot \cot\left(\dfrac{r\pi}{m}\right)\\
   &+ \sum_{l=1 }^{m-1} \cos\left( \frac{2\pi r l}{m} \right)  \zeta''\left(0, \frac{l}{m} \right)+ \pi \sum_{l=1}^{m-1} \sin\left( \frac{2\pi r l}{m} \right) \log \Gamma\left( \frac{l}{m} \right)+ (\gamma + \log (2\pi m))\sum_{l=1}^{m-1} \cos\left(\dfrac{2\pi l r}{m}\right) \log\sin\left(\dfrac{\pi l}{m}\right) 
.\nonumber
\end{align}

The second generalized Stieltjes constant is much less studied than the first constant. Similarly to the first generalized Stieltjes constant, the second generalized Stieltjes constant at rational argument may be evaluated via the following formula \cite[Page 577]{Blagouchine:2015}:
\begin{align}\label{Second-Stel-Rational-2}
\gamma_2\left(\dfrac{r}{m}\right) &= \gamma_2 - 2\gamma_1 \log m+ \gamma^2 \log 2+ \gamma (\log^2 m+2\log^2 2+  2\log 2\cdot \log (\pi m))- \frac{\pi^2}{12} \log 2 \\
&+(\log^22+\log^2m+\log^2\pi+2\log\pi \log m+2\log 2\log (\pi m))\log 2+ \dfrac{\log^3m}{3}\nonumber\\
&+ \frac{2}{3}\sum_{l=1}^{m-1} \cos\left(\frac{2\pi rl}{m}\right) \zeta'''\left(0, \frac{l}{m}\right) - 2(\gamma + \ln (2\pi m)) \sum_{l=1}^{m-1} \cos\left(\frac{2\pi rl}{m}\right) \zeta''\left(0, \frac{l}{m}\right)\nonumber\\
& + \pi \sum_{l=1}^{m-1} \sin\left(\frac{2\pi rl}{m}\right) \zeta''\left(0, \frac{l}{m}\right) - 2\pi(\gamma + \ln (2\pi m)) \sum_{l=1}^{m-1} \sin\left(\frac{2\pi rl}{m}\right) \log\Gamma\left(\frac{l}{m}\right) \nonumber\\
 & +\left(\dfrac{\pi^2}{12}-(\gamma+\log(2\pi m))^2\right) \sum_{l=1}^{m-1} \cos\left(\frac{2\pi rl}{m}\right) \log\sin\left(\dfrac{\pi l}{m}\right) + \left(\dfrac{\pi^2}{12}+(\gamma+\log(2\pi m))^2\right)\frac{\pi}{2}\cot\left(\frac{\pi r}{m}\right). \nonumber\\
\end{align}
According to Theorem \ref{thm:02}, to obtain explicit formula for $\gamma_0(f,x)$, $\gamma_1(f,x)$ and $\gamma_2(f,x)$ for rational numbers $x$ we need to give an explicit formula for $\displaystyle\sum_{r=0}^{q-1} f(r)\Psi\left(\dfrac{x+r}{q}\right)$,  $\displaystyle\sum_{r=0}^{q-1} f(r)  \gamma_1\left(\dfrac{x+r}{q}\right)$ and $\displaystyle\sum_{r=0}^{q-1} f(r)  \gamma_2\left(\dfrac{x+r}{q}\right)$ for rational numbers $0<x\leq q$. We have the following lemmas.
\begin{lemma}\label{lem:sum_f(r)Psi}
For $x$ rational with $0<x\leq q$, we have
\begin{eqnarray}
\sum_{r=0}^{q-1} f(r) \Psi \left(\dfrac{x+r}{q}\right) =& \displaystyle  (-\gamma-\log (2q)) \sum_{\substack{r=0\\ r\neq q-\lfloor x\rfloor} }^{q-1}f(r)-\dfrac{\pi}{2}\sum_{\substack{r=0\\ r\neq q-\lfloor x\rfloor} }^{q-1}f(r)\cot\left(\dfrac{(x+r)\pi}{q}\right)\\
  &+\displaystyle\sum_{l=1}^{q-1} \log\sin\left(\dfrac{\pi l}{q}\right) \sum_{\substack{r=0\\ r\neq q-\lfloor x\rfloor} }^{q-1}f(r)\cos\left(\dfrac{2\pi l (x+r)}{q}\right) \\
   &+\displaystyle f(q-\lfloor x\rfloor)\Psi (1+\{x\})+q \sum_{q-\lfloor x\rfloor+1\leq r\leq q-1}^{\prime} \dfrac{f(r)}{x+r-q}
\end{eqnarray}
\end{lemma}
In  whole of this paper the sum  $\sum^{\prime}$ is vanish if the summation set if empty. 

\begin{proof}
Using \eqref{recu-Psi} and \eqref{Psi-Rational} we have
\begin{align*}
\sum_{r=0}^{q-1} f(r) \Psi \left(\dfrac{x+r}{q}\right)= &\sum_{r=0}^{q-1} f(r) \Psi \left(\dfrac{\lfloor x\rfloor+\{x\} +r}{q}\right)\\
 =& \sum_{0\leq r\leq q-\lfloor x\rfloor-1}^{\prime} f(r) \Psi\left(\dfrac{\lfloor x\rfloor+\{x\}+r}{q}\right)+ f(q-\lfloor x\rfloor)\Psi (1+\{x\})+ \sum_{q-\lfloor x\rfloor+1\leq r\leq q-1}^{\prime} f(r) \Psi\left(\dfrac{\lfloor x\rfloor+\{x\}+r}{q}\right).
\end{align*}
If $0\leq r \leq q-\lfloor x\rfloor -1$, then $0<\lfloor x\rfloor+\{x\}\leq \lfloor x\rfloor+\{x\}+r\leq q+\{x\}-1 <q$ and we so we apply \eqref{Psi-Rational} directly.

Now, if $q-\lfloor x\rfloor +1\leq r \leq q-1$, then $0<q+\{x\}+1\leq \lfloor x\rfloor+\{x\}+r\leq q+\lfloor x\rfloor+\{x\}-1 \leq 2q-2$ and we so we can no apply \eqref{Psi-Rational} directly. But this implies $0<\{x\}+1\leq \lfloor x\rfloor+\{x\}+r-q\leq  q-2<q$. Thus, we have
\begin{align*}
\sum_{r=0}^{q-1} f(r) \Psi \left(\dfrac{x+r}{q}\right)= & \sum_{0\leq r\leq q-\lfloor x\rfloor-1}^{\prime} f(r) \Psi\left(\dfrac{x+r}{q}\right)+ f(q-\lfloor x\rfloor)\Psi (1+\{x\})+\sum_{q-\lfloor x\rfloor+1\leq r\leq q-1}^{\prime} f(r) \Psi\left(1+\dfrac{x+r-q}{q}\right)\\
 =& \sum_{0\leq r\leq q-\lfloor x\rfloor-1}^{\prime} f(r) \Psi\left(\dfrac{x+r}{q}\right)+ f(q-\lfloor x\rfloor)\Psi (1+\{x\})+\sum_{q-\lfloor x\rfloor+1\leq r\leq q-1}^{\prime} f(r) \Psi\left(\dfrac{x+r-q}{q}\right)\\
  &+q\sum_{q-\lfloor x\rfloor+1\leq r\leq q-1}^{\prime}  \dfrac{f(r)}{x+r-q}\\
 =& \sum_{0\leq r\leq q-\lfloor x\rfloor-1}^{\prime} f(r) \left[ -\gamma-\log (2q)-\dfrac{\pi}{2}\cot\left(\dfrac{(x+r)\pi}{q}\right)+\sum_{l=1}^{q-1} \cos\left(\dfrac{2\pi l (x+r)}{q}\right) \log\sin\left(\dfrac{\pi l}{q}\right)\right]\\
 & +\sum_{q-\lfloor x\rfloor+1\leq r\leq q-1}^{\prime} f(r) \left[ -\gamma-\log (2q)-\dfrac{\pi}{2}\cot\left(\dfrac{(x+r-q)\pi}{q}\right)+\sum_{l=1}^{q-1} \cos\left(\dfrac{2\pi l (x+r-q)}{q}\right) \log\sin\left(\dfrac{\pi l}{q}\right)\right]\\
  &+ f(q-\lfloor x\rfloor)\Psi (1+\{x\})+q\sum_{q-\lfloor x\rfloor+1\leq r\leq q-1}^{\prime}  \dfrac{f(r)}{x+r-q}\\
   =& \sum_{\substack{r=0\\ r\neq q-\lfloor x\rfloor} }^{q-1}f(r) \left[ -\gamma-\log (2q)-\dfrac{\pi}{2}\cot\left(\dfrac{(x+r)\pi}{q}\right)+\sum_{l=1}^{q-1} \cos\left(\dfrac{2\pi l (x+r)}{q}\right) \log\sin\left(\dfrac{\pi l}{q}\right)\right]\\
  &+ f(q-\lfloor x\rfloor)\Psi (1+\{x\})+q\sum_{q-\lfloor x\rfloor+1\leq r\leq q-1}^{\prime}  \dfrac{f(r)}{x+r-q}\\
    =&  (-\gamma-\log (2q)) \sum_{\substack{r=0\\ r\neq q-\lfloor x\rfloor} }^{q-1}f(r)-\dfrac{\pi}{2}\sum_{\substack{r=0\\ r\neq q-\lfloor x\rfloor} }^{q-1}f(r)\cot\left(\dfrac{(x+r)\pi}{q}\right)+\sum_{l=1}^{q-1} \log\sin\left(\dfrac{\pi l}{q}\right) \sum_{\substack{r=0\\ r\neq q-\lfloor x\rfloor} }^{q-1}f(r)\cos\left(\dfrac{2\pi l (x+r)}{q}\right) \\
  &+ f(q-\lfloor x\rfloor)\Psi (1+\{x\})+q\sum_{q-\lfloor x\rfloor+1\leq r\leq q-1}^{\prime}  \dfrac{f(r)}{x+r-q}.  
\end{align*}  
\end{proof}
\begin{lemma}\label{lem:sum_f(r)gamma_1}
For $x$ rational with $0<x\leq q$, we have
\begin{eqnarray}
\sum_{r=0}^{q-1}f(r)\gamma_1\left(\dfrac{x+r}{q}\right) &=&\left(\gamma_1- \gamma \log (2 q)  -\log^2 (2) - \log 2\cdot \log (\pi q) - \frac{\log^2 q}{2}\right)\sum_{\substack{r=0\\ r\neq q-\lfloor x\rfloor }}^{q-1}f(r) \\
    & &-\dfrac{\pi }{2}(\gamma + \log (2\pi q)) \sum_{\substack{r=0\\ r\neq q-\lfloor x\rfloor }}^{q-1}f(r) \cot\left(\dfrac{(x+r)\pi}{q}\right)\nonumber\\
    & & + \sum_{l=1 }^{q-1}\zeta''\left(0, \frac{l}{q} \right) \sum_{\substack{r=0\\ r\neq q-\lfloor x\rfloor }}^{q-1}f(r) \cos\left( \frac{2\pi (x+r) l}{q} \right)  \nonumber\\
   & &+ \pi \sum_{l=1}^{q-1} \log \Gamma\left( \frac{l}{q} \right) \sum_{\substack{r=0\\ r\neq q-\lfloor x\rfloor }}^{q-1}f(r) \sin\left( \frac{2\pi (x+r) l}{q} \right)\nonumber\\
   & &+ (\gamma + \log (2\pi q))\sum_{l=1}^{q-1}\log\sin\left(\dfrac{\pi l}{q}\right) \sum_{\substack{r=0\\ r\neq q-\lfloor x\rfloor }}^{q-1}f(r)\cos\left(\dfrac{2\pi l (x+r)}{q}\right)  \nonumber\\
   & &+q \sum_{q-\lfloor x\rfloor+1\leq r\leq q-1}^{\prime} \dfrac{f(r) \log\left(\frac{x+r-q}{q}\right)}{x+r-q}+ f(q-\lfloor x\rfloor)\gamma_1 (1+\{x\})\nonumber
\end{eqnarray}
\end{lemma}
\begin{proof}
\begin{align*}
\sum_{r=0}^{q-1} f(r) \gamma_1 \left(\dfrac{x+r}{q}\right)=& \sum_{0\leq r\leq q-\lfloor x\rfloor-1}^{\prime} f(r) \gamma_1\left(\dfrac{x+r}{q}\right)+ f(q-\lfloor x\rfloor)\gamma_1 (1+\{x\})+\sum_{q-\lfloor x\rfloor+1\leq r\leq q-1}^{\prime} f(r)\gamma_1\left(1+\dfrac{x+r-q}{q}\right)\\
 =&  \sum_{0\leq r\leq q-\lfloor x\rfloor-1}^{\prime} f(r) \gamma_1\left(\dfrac{x+r}{q}\right)+ f(q-\lfloor x\rfloor)\gamma_1 (1+\{x\})+\sum_{q-\lfloor x\rfloor+1\leq r\leq q-1}^{\prime} f(r)\gamma_1\left(\dfrac{x+r-q}{q}\right)\\
  & +q\sum_{q-\lfloor x\rfloor+1\leq r\leq q-1}^{\prime} \dfrac{f(r) \log\left(\frac{x+r-q}{q}\right)}{x+r-q} \\
  =& \sum_{0\leq r\leq q-\lfloor x\rfloor-1}^{\prime} f(r) \left[\gamma_1- \gamma \log (2 q)  -\log^2 (2) - \log 2\cdot \log (\pi q) - \frac{\log^2 q}{2} -\dfrac{\pi }{2}(\gamma + \log (2\pi q)) \cdot \cot\left(\dfrac{(x+r)\pi}{q}\right) \right.\\
   &+ \sum_{l=1 }^{q-1} \cos\left( \frac{2\pi (x+r) l}{q} \right)  \zeta''\left(0, \frac{l}{q} \right)+ \pi \sum_{l=1}^{q-1} \sin\left( \frac{2\pi (x+r) l}{q} \right) \log \Gamma\left( \frac{l}{q} \right)\\
   &+\left. (\gamma + \log (2\pi q))\sum_{l=1}^{q-1} \cos\left(\dfrac{2\pi l (x+r)}{q}\right) \log\sin\left(\dfrac{\pi l}{q}\right) \right] + f(q-\lfloor x\rfloor)\gamma_1 (1+\{x\})\\
   & +\sum_{q-\lfloor x\rfloor+1\leq r\leq q-1}^{\prime} f(r) \left[\gamma_1- \gamma \log (2 q)  -\log^2 (2) - \log 2\cdot \log (\pi q) - \frac{\log^2 q}{2} -\dfrac{\pi }{2}(\gamma + \log (2\pi q)) \cdot \cot\left(\dfrac{(x+r-q)\pi}{q}\right) \right.\\
   &+ \sum_{l=1 }^{q-1} \cos\left( \frac{2\pi (x+r-q) l}{q} \right)  \zeta''\left(0, \frac{l}{q} \right)+ \pi \sum_{l=1}^{q-1} \sin\left( \frac{2\pi (x+r-q) l}{q} \right) \log \Gamma\left( \frac{l}{q} \right)\\
   &+\left. (\gamma + \log (2\pi q))\sum_{l=1}^{q-1} \cos\left(\dfrac{2\pi l (x+r-q)}{q}\right) \log\sin\left(\dfrac{\pi l}{q}\right) \right]+q\sum_{q-\lfloor x\rfloor+1\leq r\leq q-1}^{\prime} \dfrac{f(r) \log\left(\frac{x+r-q}{q}\right)}{x+r-q}\\
   =& \left(\gamma_1- \gamma \log (2 q)  -\log^2 (2) - \log 2\cdot \log (\pi q) - \frac{\log^2 q}{2}\right)\sum_{\substack{r=0\\ r\neq q-\lfloor x\rfloor }}^{q-1}f(r) -\dfrac{\pi }{2}(\gamma + \log (2\pi q)) \sum_{\substack{r=0\\ r\neq q-\lfloor x\rfloor }}^{q-1}f(r) \cot\left(\dfrac{(x+r)\pi}{q}\right)\\
    & + \sum_{l=1 }^{q-1}\zeta''\left(0, \frac{l}{q} \right) \sum_{\substack{r=0\\ r\neq q-\lfloor x\rfloor }}^{q-1}f(r) \cos\left( \frac{2\pi (x+r) l}{q} \right)  + \pi \sum_{l=1}^{q-1} \log \Gamma\left( \frac{l}{q} \right) \sum_{\substack{r=0\\ r\neq q-\lfloor x\rfloor }}^{q-1}f(r) \sin\left( \frac{2\pi (x+r) l}{q} \right)\\
   &+ (\gamma + \log (2\pi q))\sum_{l=1}^{q-1}\log\sin\left(\dfrac{\pi l}{q}\right) \sum_{\substack{r=0\\ r\neq q-\lfloor x\rfloor }}^{q-1}f(r)\cos\left(\dfrac{2\pi l (x+r)}{q}\right)  +q\sum_{q-\lfloor x\rfloor+1\leq r\leq q-1}^{\prime} \dfrac{f(r) \log\left(\frac{x+r-q}{q}\right)}{x+r-q}\\
   & + f(q-\lfloor x\rfloor)\gamma_1 (1+\{x\})
\end{align*}
\end{proof}
Using similar arguments, we obtain
\begin{lemma}\label{lem:sum_f(r)gamma_2}
For $x$ rational with $0<x\leq q$, we have
\begin{align}
\sum_{r=0}^{q-1}f(r)\gamma_2\left(\dfrac{x+r}{q}\right) &=\left(\gamma_2 - 2\gamma_1 \log q+ \gamma^2 \log 2+ \gamma (\log^2 q+2\log^2 2+  2\log 2\cdot \log (\pi q))- \frac{\pi^2}{12} \log 2 \right.\\
&\left. +(\log^22+\log^2q+\log^2\pi+2\log\pi \log q+2\log 2\log (\pi q))\log 2+ \dfrac{\log^3q}{3}\right) \sum_{\substack{r=0\\ r\neq q-\lfloor x\rfloor }}^{q-1}f(r)\nonumber\\
&+ \frac{2}{3}\sum_{l=1}^{q-1}\zeta'''\left(0, \frac{l}{q}\right) \sum_{\substack{r=0\\ r\neq q-\lfloor x\rfloor }}^{q-1}f(r)\cos\left(\frac{2\pi (x+r)l}{q}\right)\nonumber\\
 &- 2(\gamma + \ln (2\pi q)) \sum_{l=1}^{q-1}\zeta''\left(0, \frac{l}{q}\right) \sum_{\substack{r=0\\ r\neq q-\lfloor x\rfloor }}^{q-1}f(r)\cos\left(\frac{2\pi (x+r)l}{q}\right) \nonumber\\
& + \pi \sum_{l=1}^{q-1}\zeta''\left(0, \frac{l}{q}\right) \sum_{\substack{r=0\\ r\neq q-\lfloor x\rfloor }}^{q-1}f(r)\sin\left(\frac{2\pi (x+r)l}{q}\right) \nonumber\\
& - 2\pi(\gamma + \ln (2\pi q)) \sum_{l=1}^{q-1}\log\Gamma\left(\frac{l}{q}\right) \sum_{\substack{r=0\\ r\neq q-\lfloor x\rfloor }}^{q-1}f(r) \sin\left(\frac{2\pi (x+r)l}{q}\right) \nonumber\\
 & +\left(\dfrac{\pi^2}{12}-(\gamma+\log(2\pi q))^2\right) \sum_{l=1}^{q-1} \log\sin\left(\dfrac{\pi l}{q}\right) \sum_{\substack{r=0\\ r\neq q-\lfloor x\rfloor }}^{q-1}f(r)\cos\left(\frac{2\pi (x+r)l}{q}\right)\nonumber \\
 & + \left(\dfrac{\pi^2}{12}+(\gamma+\log(2\pi q))^2\right)\frac{\pi}{2}\sum_{\substack{r=0\\ r\neq q-\lfloor x\rfloor }}^{q-1}f(r)\cot\left(\frac{\pi (x+r)}{q}\right) \nonumber\\
 &+q\sum_{q-\lfloor x\rfloor+1\leq r\leq q-1}^{\prime} \dfrac{f(r) \log^2\left(\frac{x+r-q}{q}\right)}{x+r-q}+ f(q-\lfloor x\rfloor)\gamma_2 (1+\{x\}).\nonumber
\end{align}
\end{lemma}

Using \eqref{gamma_0(r,q,x)-versus-gamma} and Lemma \ref{lem:sum_f(r)Psi} we get
\begin{theorem}\label{thm:gamma_0(f,x)}
For $x$ rational with $0<x\leq q$, we have
\begin{eqnarray}
\gamma_0(f,x) &=& \displaystyle \dfrac{\gamma+\log 2}{q} \sum_{\substack{r=0\\ r\neq q-\lfloor x\rfloor} }^{q-1}f(r)+\dfrac{\pi}{2q}\sum_{\substack{r=0\\ r\neq q-\lfloor x\rfloor} }^{q-1}f(r)\cot\left(\dfrac{(x+r)\pi}{q}\right)\\
  & &-\dfrac{1}{q}\displaystyle\sum_{l=1}^{q-1} \log\sin\left(\dfrac{\pi l}{q}\right) \sum_{\substack{r=0\\ r\neq q-\lfloor x\rfloor} }^{q-1}f(r)\cos\left(\dfrac{2\pi l (x+r)}{q}\right) \nonumber\\
  & &-\displaystyle \dfrac{f(-\lfloor x\rfloor)}{q} \left(\Psi (1+\{x\})+\log q\right)-\sum_{q-\lfloor x\rfloor+1\leq r\leq q-1}^{\prime}  \dfrac{f(r)}{x+r-q}.\nonumber
\end{eqnarray}
\end{theorem}

Using \eqref{gamma_1(r,q,x)-versus-gamma} and Lemmas \ref{lem:sum_f(r)Psi} and \ref{lem:sum_f(r)gamma_1}  we obtain
\begin{theorem}\label{thm:gamma_1(f,x)}
For $x$ rational with $0<x\leq q$, we have
\begin{eqnarray}
\gamma_1(f,x) &=&\displaystyle  \dfrac{\log 2}{q}  \left(\dfrac{\gamma_1}{\log 2}-\gamma- \log 2\pi\right) \sum_{\substack{r=0\\ r\neq q-\lfloor x\rfloor} }^{q-1}f(r)-\dfrac{\pi}{2q} (\gamma+\log 2\pi)\sum_{\substack{r=0\\ r\neq q-\lfloor x\rfloor} }^{q-1}f(r)\cot\left(\dfrac{(x+r)\pi}{q}\right)\\
  & &\dfrac{1}{q}(\gamma+\log 2\pi) \displaystyle\sum_{l=1}^{q-1} \log\sin\left(\dfrac{\pi l}{q}\right) \sum_{\substack{r=0\\ r\neq q-\lfloor x\rfloor} }^{q-1}f(r)\cos\left(\dfrac{2\pi l (x+r)}{q}\right) \nonumber\\
  & & +\dfrac{1}{q} \sum_{l=1 }^{q-1}\zeta''\left(0, \frac{l}{q} \right) \sum_{\substack{r=0\\ r\neq q-\lfloor x\rfloor }}^{q-1}f(r) \cos\left( \frac{2\pi (x+r) l}{q} \right)  \nonumber\\
   & &+\dfrac{\pi}{q}  \sum_{l=1}^{q-1} \log \Gamma\left( \frac{l}{q} \right) \sum_{\substack{r=0\\ r\neq q-\lfloor x\rfloor }}^{q-1}f(r) \sin\left( \frac{2\pi (x+r) l}{q} \right)\nonumber\\
   & &+\sum_{q-\lfloor x\rfloor+1\leq r\leq q-1}^{\prime} \dfrac{f(r) \log\left(\frac{x+r-q}{q^2}\right)}{x+r-q}+ \dfrac{f(-\lfloor x\rfloor)}{q} \left(\gamma_1 (1+\{x\}) -\log q \Psi (1+\{x\})-\dfrac{\log^2 q}{2}\right)\nonumber
\end{eqnarray}
\end{theorem}
Using \eqref{gamma_2(r,q,x)-versus-gamma} and Lemmas \ref{lem:sum_f(r)Psi}, \ref{lem:sum_f(r)gamma_1} and \ref{lem:sum_f(r)gamma_2} we obtain
\begin{theorem}\label{thm:gamma_2(f,x)}
For $x$ rational with $0<x\leq q$, we have
\begin{eqnarray}
\gamma_2(f,x) &=& \displaystyle \dfrac{\log 2}{q}\left( \dfrac{\gamma_2}{\log 2} +\gamma^2 + 2\gamma \log (2\pi )- \frac{\pi^2}{12} +(\log 2+\log \pi )^2\right) \sum_{\substack{r=0\\ r\neq q-\lfloor x\rfloor }}^{q-1}f(r)\\
  &  &+ \dfrac{\pi}{2q} \left(\dfrac{\pi^2}{12}+(\gamma+\log(2\pi))^2\right)\sum_{\substack{r=0\\ r\neq q-\lfloor x\rfloor }}^{q-1}f(r)\cot\left(\frac{\pi (x+r)}{q}\right) \nonumber\\
  & &+\dfrac{1}{q} \left(\dfrac{\pi^2}{12}-(\gamma+\log 2\pi)^2\right)\displaystyle\sum_{l=1}^{q-1} \log\sin\left(\dfrac{\pi l}{q}\right) \sum_{\substack{r=0\\ r\neq q-\lfloor x\rfloor} }^{q-1}f(r)\cos\left(\dfrac{2\pi l (x+r)}{q}\right) \nonumber\\
  & & -\dfrac{2}{q} (\gamma+\log (2\pi)) \sum_{l=1 }^{q-1}\zeta''\left(0, \frac{l}{q} \right) \sum_{\substack{r=0\\ r\neq q-\lfloor x\rfloor }}^{q-1}f(r) \cos\left( \frac{2\pi (x+r) l}{q} \right)  \nonumber\\
 & &-\dfrac{2\pi}{q}  (\gamma + \ln (2\pi )) \sum_{l=1}^{q-1}\log\Gamma\left(\frac{l}{q}\right) \sum_{\substack{r=0\\ r\neq q-\lfloor x\rfloor }}^{q-1}f(r) \sin\left(\frac{2\pi (x+r)l}{q}\right) \nonumber\\
   & &+ \frac{2}{3q}\sum_{l=1}^{q-1}\zeta'''\left(0, \frac{l}{q}\right) \sum_{\substack{r=0\\ r\neq q-\lfloor x\rfloor }}^{q-1}f(r)\cos\left(\frac{2\pi (x+r)l}{q}\right)\nonumber\\
 & & +\dfrac{\pi}{q}  \sum_{l=1}^{q-1}\zeta''\left(0, \frac{l}{q}\right) \sum_{\substack{r=0\\ r\neq q-\lfloor x\rfloor }}^{q-1}f(r)\sin\left(\frac{2\pi (x+r)l}{q}\right) \nonumber\\
 & &-\log^2 q\sum_{q-\lfloor x\rfloor+1\leq r\leq q-1}^{\prime}  \dfrac{f(r)}{x+r-q} +\sum_{q-\lfloor x\rfloor+1\leq r\leq q-1}^{\prime} \dfrac{f(r) \log^2\left(x+r-q\right)}{x+r-q}\nonumber\\
 & &+ \dfrac{f(-\lfloor x\rfloor)}{q} \left(\gamma_2 (1+\{x\})+2\log q \gamma_1(1+\{x\})-\log^2 q \Psi (1+\{x\})+\dfrac{\log^3q}{q}  \right).\nonumber   
\end{eqnarray}
\end{theorem}
Using \eqref{eq:gamma_k(f)-vesus-gamma-k(f-x)} and Theorems \ref{thm:gamma_0(f,x)}, \ref{thm:gamma_1(f,x)} and  \ref{thm:gamma_2(f,x)}  we get
\begin{theorem}\label{lem:gamma(f)}
We have
\begin{eqnarray}
\gamma_0(f)&=& \displaystyle\dfrac{f(0)}{q} (\gamma-\log q)+  \dfrac{1}{q}(\gamma+\log 2) \sum_{r=1}^{q-1} f(r)  +\dfrac{\pi}{2q} \sum_{r=1}^{q-1} f(r)\cot \left(\dfrac{r\pi}{q}\right)\\
 & &\displaystyle -\dfrac{1}{q} \sum_{l=1}^{q-1}  \log\sin \left(\dfrac{\pi l}{q}\right)\left( \sum_{r=1}^{q-1} f(r)\cos\left( \dfrac{2\pi lr}{q}\right)\right).\nonumber
\end{eqnarray}
\begin{eqnarray}
\gamma_1(f)&=&\dfrac{f(0)}{q} \left(\gamma_1+\gamma\log q- \log^2 q \right)+ \dfrac{1}{q}\left(\gamma_1-\gamma\log 2+\log 2 \cdot \log (2\pi) \right)  \sum_{r=1}^{q-1} f(r)\\
 & &-\dfrac{\pi}{2q} (\gamma+\log (2\pi)) \sum_{r=1}^{q-1} f(r) \cot \left( \dfrac{r \pi}{q}\right)\nonumber\\
 & & +\dfrac{1}{q} (\gamma+\log (2\pi)) \sum_{l=1}^{q-1} \log\sin\left(\dfrac{l}{q}\right)  \sum_{r=1}^{q-1}f(r) \cos \left( \dfrac{2\pi lr }{q}\right)\nonumber\\
   & &+\dfrac{\pi}{q}  \sum_{l=1}^{q-1} \log \Gamma\left( \frac{l}{q} \right)\sum_{r=1}^{q-1}f(r) \sin\left( \frac{2\pi r l}{q} \right)  
+\dfrac{1}{q} \sum_{l=1 }^{q-1} \zeta''\left(0, \frac{l}{q} \right) \sum_{r=1}^{q-1}f(r)\cos\left( \frac{2\pi r l}{q} \right).\nonumber
\end{eqnarray}
\begin{eqnarray}
\gamma_2(f) 
 &=&\dfrac{f(0)}{q} \left(\gamma_2+2\gamma_1\log q +\gamma\log^2 q- \dfrac{\log^3 q}{3} \right)\nonumber\\
 & &+ \dfrac{1}{q}\left(\gamma_2+\gamma^2\log 2+ 2\gamma\log 2 \log (2\pi)-\left(\dfrac{\pi^2}{12}-\log^2(2\pi)\right) \log 2  \right)  \sum_{r=1}^{q-1} f(r) \nonumber\\
 & &-\dfrac{\pi}{2q} \left(\log^2q-2\log q(\gamma+\log (2\pi q))+\dfrac{\pi^2}{12}+(\gamma+\log (2\pi q))^2\right) \sum_{r=1}^{q-1} f(r) \cot \left( \dfrac{r \pi}{q}\right)\nonumber\\
 & & +\dfrac{1}{q} \left(-\log^2 q+2\log q (\gamma+\log (2\pi q)+ \dfrac{\pi^2}{12})-(\gamma+\log (2\pi q))^2\right) \sum_{l=1}^{q-1} \log\sin\left(\dfrac{l}{q}\right)  \sum_{r=1}^{q-1}f(r) \cos \left( \dfrac{2\pi lr }{q}\right)\nonumber\\
   & &+\dfrac{2\pi}{q} \left(\log q-\pi (\gamma+\log (2\pi q))\right)  \sum_{l=1}^{q-1} \log \Gamma\left( \frac{l}{q} \right)\sum_{r=1}^{q-1}f(r) \sin\left( \frac{2\pi r l}{q} \right) \nonumber\\
   & &-\dfrac{2}{q} \left(\gamma+\log (2\pi ) \right) \sum_{l=1 }^{q-1} \zeta''\left(0, \frac{l}{q} \right) \sum_{r=1}^{q-1}f(r)\cos\left( \frac{2\pi r l}{q} \right) + \frac{2}{3q}\sum_{l=1}^{q-1}\zeta'''\left(0, \frac{l}{q}\right)\sum_{r=1}^{q-1}f(r) \cos\left(\frac{2\pi rl}{q}\right)\nonumber\\
   & &+\dfrac{\pi}{q}  \sum_{l=1}^{q-1} \zeta''\left(0, \frac{l}{q}\right)\sum_{r=1}^{q-1}f(r) \sin\left(\frac{2\pi rl}{q}\right).\nonumber 
\end{eqnarray}
\end{theorem}
\section{Examples}\label{section 7}
In this section we give interesting  computations on Stieltjes constants  of $L$-functions associated to different arithmetic functions. 

\subsection{Alternating Hurwitz zeta function}
In this case, $f(n)=(-1)^n$. So we obtain $q=2$. Thus, we calculate the Stieltjes constants for $0<x\leq 2$. We have
\begin{itemize}
\item If $0<x<1$,
$$
\gamma_0(f,x)=- \dfrac{\log 2}{2}-\dfrac{\Psi (1+x)}{2} +\dfrac{\pi}{4}\cot\left(\dfrac{x\pi}{2}\right)  + \dfrac{\pi}{4\cot\left(\dfrac{x\pi}{2}\right)},
$$
$$
\begin{array}{ll}
\gamma_1(f,x)=& -\dfrac{\pi}{4}(\gamma+\log 2\pi)\cot\left(\dfrac{x\pi}{2}\right)+\dfrac{\pi(\gamma+\log 2\pi)}{4\cot\left(\dfrac{x\pi}{2}\right)}\\
& \left. -\left(\dfrac{\log^2 2}{2}-\log 2 \cdot\log 2\pi \right) \cos (\pi x)+\dfrac{\pi}{2}  \log \sqrt{\pi}  \sin (\pi x) \right. \\
& +\dfrac{\gamma_1 (1+x)}{2} -\dfrac{\log 2}{2} \cdot  \Psi (1+x)-\dfrac{\log^2 2}{4},
\end{array}
$$
and
$$
\begin{array}{ll}
\gamma_2(f,x)= &+ \dfrac{\pi}{4} \left(\dfrac{\pi^2}{12}+(\gamma+\log(2\pi))^2\right)\left( \cot\left(\frac{\pi x}{2}\right)+\dfrac{1}{\cot\left(\frac{\pi x}{2}\right)}\right) \\
& +2\left[ (\gamma+\log (2\pi))\left(\dfrac{\log^2 2}{2}+ \log 2 \cdot \log 2\pi\right)  \cos\left( \pi x \right)  \right.  \\
& + \left( -\dfrac{\log^3 2}{2}-\dfrac{3}{2} \log^2 2 \cdot \log 2\pi +3 \log 2 \left(\dfrac{\gamma^2}{2}-\dfrac{\pi^2}{24}-\dfrac{\log^2 2\pi}{2}+ \gamma_1 \right)\right) \cos\left( \pi x \right)\\
& -\pi  (\gamma + \ln (2\pi )) \log \sqrt{\pi} \sin\left( \pi x \right)   \\
&\left. -\dfrac{\pi}{2}  \left(\dfrac{\log^2 2}{2}+\log 2 \cdot \log 2\pi \right) \sin\left( \pi x \right)\right]  \\
& + \dfrac{1}{2} \left(\gamma_2 (1+x)+2\log 2 \gamma_1(1+x)-\log^2 2 \Psi (1+x)+\dfrac{\log^3 2}{2}  \right).
\end{array}
$$
\item If $1\leq x<2$,
$$
\gamma_0(f,x)=\dfrac{\gamma}{2}+\log 2  +\dfrac{\pi}{4}\cot\left(\dfrac{x\pi}{2}\right) +\dfrac{\Psi (x)}{2},
$$
$$
\begin{array}{ll}
\gamma_1(f,x)= & \gamma_1-\dfrac{\gamma \log 2}{2}-\dfrac{\log 2 \cdot\log 2\pi}{2} -\dfrac{\pi}{4}(\gamma+\log 2\pi)\cot\left(\dfrac{x\pi}{2}\right) \\
& -\left(\dfrac{\log^2 2}{4}-\dfrac{\log 2 \cdot\log 2\pi}{2}  \right) \cos (\pi x)+\dfrac{\pi}{2}  \log \sqrt{\pi}  \sin (\pi x) \\
&  - \dfrac{\gamma_1 (x)}{2} +\dfrac{\log 2}{2} \cdot  \Psi (x)+\dfrac{\log^2 2}{4},
\end{array}
$$
and
$$
\begin{array}{ll}
\gamma_2(f,x)= & 
\displaystyle \dfrac{\log 2}{2}\left( \dfrac{\gamma_2}{\log 2} +\gamma^2 + 2\gamma \log (2\pi )- \frac{\pi^2}{12} +(\log 2+\log \pi )^2\right)   \\
&+ \dfrac{\pi}{4} \left(\dfrac{\pi^2}{12}+(\gamma+\log(2\pi))^2\right)\cot\left(\frac{\pi x}{2}\right)   \\
&+ (\gamma+\log (2\pi))\left(\dfrac{\log^2 2}{2}+ \log 2 \cdot \log 2\pi\right)  \cos\left( \pi x \right)   \\
&-\pi  (\gamma + \ln (2\pi )) \log \sqrt{\pi} \sin\left( \pi x \right)   \\
& + \left( -\dfrac{\log^3 2}{2}-\dfrac{3}{2} \log^2 2 \cdot \log 2\pi +3 \log 2 \left(\dfrac{\gamma^2}{2}-\dfrac{\pi^2}{24}-\dfrac{\log^2 2\pi}{2}+ \gamma_1 \right)\right)  \cos\left( \pi x \right)  \\
& -\dfrac{\pi}{2}  \left(\dfrac{\log^2 2}{2}+\log 2 \cdot \log 2\pi \right) \sin\left( \pi x \right)  \\
& - \dfrac{1}{2} \left(\gamma_2 (x)+2\log 2 \gamma_1(x)-\log^2 2 \Psi (x)+\dfrac{\log^3 2}{2}  \right),
\end{array}
$$
\item If $x=2$, then
\begin{eqnarray*}
\gamma_0(f,2)&=&-\log2 +1,\\
\gamma_1(f,2)&=&\gamma \log 2  -\dfrac{\log^2 2}{2} +  \log 2 \cdot \log 2\pi,
\end{eqnarray*}
and
$$
\begin{array}{ll}
\gamma_2(f,2)= & \displaystyle -\dfrac{\log 2}{2}\left( \dfrac{\gamma_2}{\log 2} +\gamma^2 + 2\gamma \log (2\pi )- \frac{\pi^2}{12} +(\log 2+\log \pi )^2\right)  \\
&+ (\gamma+\log (2\pi))\left(\dfrac{\log^2 2}{2}+ \log 2 \cdot \log 2\pi\right)  \\
&+ \left( -\dfrac{\log^3 2}{2}-\dfrac{3}{2} \log^2 2 \cdot \log 2\pi +3 \log 2 \left(\dfrac{\gamma^2}{2}-\dfrac{\pi^2}{24}-\dfrac{\log^2 2\pi}{2}+ \gamma_1 \right)\right)   \\
&+ \dfrac{1}{2} \left(\gamma_2+2\log 2 \gamma_1+\gamma\log^2 2 +\dfrac{\log^3 2}{2}  \right).
\end{array}
$$
\end{itemize}

\subsection{Principal character modulus $q$ where $q$ is a prime number}
Assume that $q$ is a prime number and let $\chi$ be principal character modulus $q$. Then, we have for $1\leq r\leq q-1$, $\chi(r)=1$ and so $\sum_{r=1}^{q-1}f(r)=q$. Using the fact that
$$
\sum_{r=1}^{q-1}\cot\left(\dfrac{\pi r}{q}\right)=0
$$
for $\theta\not\equiv 0\pmod {2\pi}$,
$$
\sum_{r=0}^{q}\sin r\theta =\dfrac{\cos \dfrac{\theta}{2}-\cos\left(\left(q+\dfrac{1}{2}\right) \theta\right)}{2\sin \dfrac{\theta}{2}}
 \quad \text{and}\quad 
\sum_{r=1}^{q}\cos r\theta =\dfrac{-\sin \dfrac{\theta}{2}+\sin\left(\left(q+\dfrac{1}{2}\right) \theta\right)}{2\sin \dfrac{\theta}{2}},
$$
we obtain
$$
\sum_{r=1}^{q-1}\cos \left(\dfrac{2\pi lr}{q}\right) =-1 \quad \text{and} \quad \sum_{r=1}^{q-1}\sin \left(\dfrac{2\pi lr}{q}\right)=0.
$$
So, we get
\begin{eqnarray}
\gamma_0(f)&=& \displaystyle  \gamma+\log 2 +\dfrac{1}{q} \sum_{l=1}^{q-1}  \log\sin \left(\dfrac{\pi l}{q}\right)= \gamma+\log 2 +\dfrac{1}{q} \log\left(\dfrac{q}{2^{q-1}}\right). 
\end{eqnarray}
\begin{eqnarray}
\gamma_1(f)&=&\gamma_1-\gamma\log 2+\log 2 \cdot \log (2\pi)+\dfrac{1}{q} (\gamma+\log (2\pi)) \log \left(\dfrac{q}{2^{q-1}}\right)-\dfrac{1}{q} \sum_{l=1 }^{q-1} \zeta''\left(0, \frac{l}{q} \right)\nonumber\\
  &=& \gamma_1-\gamma\log 2+\log 2 \cdot \log (2\pi)+\dfrac{1}{q} (\gamma+\log (2\pi)) \log \left(\dfrac{q}{2^{q-1}}\right)+\dfrac{1}{q} \left( \dfrac{\log^2 q}{2}+ \log q \cdot \log 2\pi \right).
\end{eqnarray}
\begin{eqnarray}
\gamma_2(f) & =&\gamma_2+\gamma^2\log 2+ 2\gamma\log 2 \log (2\pi)-\left(\dfrac{\pi^2}{12}-\log^2(2\pi)\right) \log 2 \nonumber\\
 & & +\dfrac{1}{q} \left(-\log^2 q+2\log q (\gamma+\log (2\pi q)+ \dfrac{\pi^2}{12})-(\gamma+\log (2\pi q))^2\right) \log \left( \dfrac{q}{2^{q-1}}\right)\nonumber\\
   & &+\dfrac{2}{q} \left(\gamma+\log (2\pi ) \right) \left( \dfrac{\log^2q}{2}-\log q\cdot \log 2\pi \right)\nonumber\\
 & & + \frac{2}{3q}\left( \dfrac{\log^3 q}{2}+\dfrac{3}{2} \log^2 q \cdot \log 2\pi -3\log q \left(\dfrac{\gamma^2}{2}-\dfrac{\pi^2}{24}-\dfrac{1}{2} \log^2(2\pi)+\gamma_1 \right) \right)\nonumber\\
\end{eqnarray}
\subsection{Nonprincipal real character modulus 4}
Let $\chi_4$ the nonprincipal real character modulus $4$. 
We have
\begin{theorem}
\begin{eqnarray}
&\gamma_0(\chi_4)&= \dfrac{\pi}{4}.\\ 
&\gamma_1(\chi_4)&=
 \dfrac{\pi}{4} \left( \log \Gamma\left(\dfrac{1}{4}\right)- \log\Gamma\left(\dfrac{3}{4}\right)\right).\\ 
&\gamma_2(\chi_4) &= \dfrac{\pi}{2} \left(\log 4-\pi (\gamma+\log (8\pi))\right)  \left( \log \Gamma\left(\dfrac{1}{4}\right)- \log\Gamma\left(\dfrac{3}{4}\right)\right)+\dfrac{\pi}{4}  \left( \zeta''\left(0, \frac{1}{4}\right) - \zeta''\left(0, \frac{3}{4}\right) \right). 
\end{eqnarray}
\end{theorem}
\subsection{Nonprincipal, real, primitive odd character modulus $q\geq 7$} 
Let $\chi$ be a nonprincipal, real, primitive odd character modulus $q\geq 7$. Denote by $h(-q)$ the class number of the imaginary quadratic field $\mathbb{Q}(\sqrt{-q})$. We have \cite[Corollary 2.3]{Berndt-Zaharescu:2004}
$$
\sum_{r=1}^{q-1} \chi (r)\cot\left(\dfrac{\pi r}{q}\right)=2\sqrt{q} h(-q).
$$
Or we have the following extension \cite[Corollary 4.5]{Beck-Halloran:2010}
$$
\sum_{r=1}^{q-1} \chi (r)\cot\left(\dfrac{\pi b r}{q}\right)=2\sqrt{q} \chi (b) h(-q)\quad \text{ where } \gcd(b,q)=1.
$$
If $b$ is a positive integer not divisible by $q$, then \cite[Corollary 4.3]{Beck-Halloran:2010}
$$
\sum_{r=1}^{q-1} \chi (r)\sin\left(\dfrac{2\pi b r}{q}\right)=\sqrt{q} \chi (b) 
\quad \text{and}\quad 
\sum_{r=1}^{q-1} \chi (r)\cos\left(\dfrac{2\pi b r}{q}\right)=0.
$$
Hence, we can express $\gamma_0(\chi)$, $\gamma_1(\chi)$ and $\gamma_2(\chi)$ in terms of $h(-q)$ as the following
\begin{theorem} Under above hypothesis we have
\begin{eqnarray}
&\gamma_0(\chi)=&\dfrac{\pi}{\sqrt{q}} h(-q),\\
&\gamma_1(\chi)=& -\dfrac{\pi}{\sqrt{q}} (\gamma+\log (2\pi)) h(-q)+ \dfrac{\pi}{\sqrt{q}} \sum_{l=1}^{q-1}\chi (l) \log \Gamma\left( \dfrac{l}{q}\right)\\ 
&\gamma_2(\chi)  =&-\dfrac{\pi}{\sqrt{q}} \left(\log^2q-2\log q(\gamma+\log (2\pi q))+\dfrac{\pi^2}{12}+(\gamma+\log (2\pi q))^2\right)  h(-q)\\
   & &+\dfrac{2\pi}{\sqrt{q}} \left(\log q-\pi (\gamma+\log (2\pi q))\right)  \sum_{l=1}^{q-1}\chi (l) \log \Gamma\left( \frac{l}{q} \right)+\dfrac{\pi}{\sqrt{q}}  \sum_{l=1}^{q-1}\chi (l) \zeta''\left(0, \frac{l}{q}\right).\nonumber 
\end{eqnarray}
\end{theorem}

\bibliographystyle{cas-model2-names}

\end{document}